\documentclass[reqno, 10pt]{amsart}
\usepackage{amssymb,mathrsfs}
\usepackage[usenames, dvipsnames]{color}
\usepackage{esint}
\usepackage{hyperref}
\usepackage{todonotes}
\usepackage{verbatim}
\usepackage{enumerate}

\usepackage{cite}
\usepackage[nobysame,non-sorted-cites,alphabetic]{amsrefs}
\def\MR#1{}

\theoremstyle{plain}
\newtheorem{theorem}{Theorem}[section]
\newtheorem{lemma}[theorem]{Lemma}

\newtheorem{proposition}[theorem]{Proposition}

\theoremstyle{definition}

\theoremstyle{remark}
\newtheorem{remark}[theorem]{Remark}

\numberwithin{equation}{section}

\newcommand{\bN}{\mathbb{N}}

\newcommand{\bZ}{\mathbb{Z}}

\newcommand\cX{\mathcal{X}}

\makeatletter
\def\dashint{\operatorname%
{\,\,\text{\bf--}\kern-.98em\DOTSI\intop\ilimits@\!\!}}
\makeatother

\begin{document}
\title[Small scales]{Small scales in inviscid limits of steady fluids}
\author[Y. Guo]{Yan Guo}

\author[Z. Yang]{Zhuolun Yang}

\address[Y. Guo]{Division of Applied Mathematics, Brown University, 182 George Street, Providence, RI 02912, USA}
\email{yan\_guo@brown.edu}

\address[Z. Yang]{Division of Applied Mathematics, Brown University, 182 George Street, Providence, RI 02912, USA}
\email{zhuolun\_yang@brown.edu}

\thanks{Y. Guo was partially supported by NSF Grant DMS-2405051}
\thanks{Z. Yang was partially supported by the AMS-Simons Travel Grant}


%\today
\begin{abstract}
In this article, we study the 2D incompressible steady Navier-Stokes equation in a channel $(-L,0)\times(-1,1)$ with the no-slip boundary condition on $\{Y = \pm 1\}$, and consider the inviscid limit $\varepsilon \to 0$. In the special case of Euler shear flow $(u_e(Y),0)$, we construct a steady Navier-Stokes solution for $\varepsilon \ll1$,
$$\left\{
\begin{aligned}
&u^\varepsilon \sim u_e + u_p + O(\sqrt{\varepsilon}),\\
&v^\varepsilon \sim  h(Y)  \exp\{Xu_e(Y)/\varepsilon\} + O(\sqrt{\varepsilon}),
\end{aligned}\right.
$$
where $u_p$ represents the classical Prandtl layer profile, and $h(Y)$ is an arbitrary smooth, compactly-supported function with small magnitude. While the classical Prandtl boundary layer $u_p$ exhibits a small scale of order $\sqrt{\varepsilon}$ in $Y$ near $Y = \pm 1$, the profile we construct reveals an $\varepsilon$ small scale of $Xu_e(Y)$ in the vertical velocity component.
\end{abstract}

\maketitle

\section{Introduction}
We study the incompressible steady Navier-Stokes equation in a two dimensional channel $\Omega = (-L,0)\times(-1,1)$
\begin{equation}\label{NS}
\left\{
\begin{aligned}
u^\varepsilon u^\varepsilon_X + v^\varepsilon u^\varepsilon_Y + p_X^\varepsilon - \varepsilon \Delta u^\varepsilon &= 0,\\
u^\varepsilon v^\varepsilon_X + v^\varepsilon v^\varepsilon_Y + p_Y^\varepsilon - \varepsilon \Delta v^\varepsilon &= 0,\\
u^\varepsilon_X + v^\varepsilon_Y &= 0,
\end{aligned}
\right.
\end{equation}
with the no-slip boundary condition on $\{Y = \pm 1\}$:
$$u^\varepsilon(X,1) = u^\varepsilon(X,-1) = v^\varepsilon(X,1) = v^\varepsilon(X,-1) = 0.$$
The boundary condition at $X=-L$ and $X=0$ are considered as inflow and outflow conditions, which will be prescribed later in the article.

We would like to study the asymptotic behavior of the solutions $(u^\varepsilon, v^\varepsilon)$ of \eqref{NS} as the viscosity $\varepsilon \to 0$. Formally, as $\varepsilon \to 0$, the equation \eqref{NS} converges to the steady Euler equation
\begin{equation}\label{Euler}
\left\{
\begin{aligned}
u^0_e u^0_{eX} + v^0_e u^0_{eY} + p_{eX}^0 &= 0,\\
u^0_e v^0_{eX} + v^0_e v^0_{eY} + p_{eY}^0 &= 0,\\
u^0_{eX} + v^0_{eY} &= 0,
\end{aligned}
\right.
\end{equation}
with no penetration boundary condition on $\{Y = \pm 1\}$:
$$
v^0_e(X,-1) = v^0_e (X,1) = 0.
$$
Due to the mismatch of the boundary value $u^\varepsilon$ and $u^0_e$ on $\{Y = \pm 1\}$, one does not expect the uniform convergence of $u^\varepsilon$ to $u^0_e$. In 1904, Prandtl proposed a scheme that the solution of \eqref{NS} can be expressed as the sum of \eqref{Euler} and a boundary layer corrector, which is a solution of the Prandtl equation. For example, near the bottom boundary $\{Y = -1\}$, Prandtl suggested that
\begin{equation}\label{Prandtl ansatz}
\begin{aligned}
u^\varepsilon(X,Y) &= u^0_e(X,Y) + u^0_p (x,y) + O(\sqrt{\varepsilon}),\\
v^\varepsilon(X,Y) &= v^0_e(X,Y) + \sqrt{\varepsilon} v^0_p(x,y) + O(\sqrt{\varepsilon}),
\end{aligned}
\end{equation}
where $(x,y)$ are the scaled variables
$$
x = X, \quad y = \frac{Y+1}{\sqrt{\varepsilon}},
$$
and $(u^0_p, v^0_p)$ is a solution to the Prandtl equation
\begin{equation}\label{Prandtl}
\left\{
\begin{aligned}
\Big(u^0_e(X,-1)+u^0_p\Big) u^0_{px} + \Big(v^0_p - v^0_p(x,0)\Big) u^0_{py} + p_{px}^0 - u^0_{pyy} &= 0,\\
p_{py}^0 &= 0,\\
u^0_{px} + v^0_{py} &= 0,
\end{aligned}
\right.
\end{equation}
with the boundary condition $u^0_{p}|_{y = 0} = - u^0_e|_{Y = -1}$.

Although the ansatz \eqref{Prandtl ansatz} was proposed over a century ago, its validity has only been rigorously justified very recently. This was first achieved by Guo and Nguyen in \cite{GuoNguyen17}, under the assumption of a moving boundary for small $x$ when the background Euler flow is a shear flow. This result has since been generalized in several directions. Iyer, in \cite{Iyer17}, studied flows over a rotating disk, established global validity in $x$ in \cites{Iyer1,Iyer2,Iyer3}, and considered non-shear background Euler flows in \cite{Iyer19}. All of these results assume that solutions of the Navier-Stokes equations satisfy a slip boundary condition. 

The validity of \eqref{Prandtl ansatz} under the classical no-slip boundary condition was recently established by Guo and Iyer \cite{GuoIyer23} for shear Euler flow in $x \in (0,L)$ for $L$ sufficiently small. This result covers an important class of Prandtl layer, the Blasius self-similar solutions. Gao and Zhang \cite{GaoZhang23} generalized this result in two directions: (1) the background Euler can be non-shear, and (2) $L$ can be large under the assumption that the Prandtl layer is concave. Meanwhile, shear Prandtl layer was consider by Gerard-Varet and Maekawa \cite{GVM19} in a narrow periodic in $x$ domain. Finally, Iyer and Masmoudi \cite{IyerMasmoudi21a} established the global in $x$ validity of \eqref{Prandtl ansatz}.

The Prandtl layer is known to exhibit a small scale in $Y$ variable near the boundaries $\{Y = \pm 1\}$, while it does not exhibit any small scale in $X$ variable. The main objective of this article is to demonstrate the existence of solutions to \eqref{NS} that are close to the ansatz \eqref{Prandtl ansatz} in the $L^\infty$ topology, and also exhibit small-scale behavior in $X$ variable near the right side boundary $\{X=0\}$. We will construct the solution in the following form:
\begin{equation}\label{expansion}
\left\{
\begin{aligned}
u^\varepsilon &= u_e^0 + u_p + u_r + \varepsilon^{13/2} u,\\
v^\varepsilon &= v_e^0 + v_p + v_r + \varepsilon^{13/2} v.
\end{aligned}
\right.
\end{equation}
Here in \eqref{expansion}, $(u_e^0,v_e^0)$ represents the background Euler flow; $(u_p,v_p)$ consists of the classical alternating Prandtl-Euler correctors; $(u_r, v_r)$  is the profile that reveals an $\varepsilon$ small scale near the right outlet: the leading order of $(u_r, v_r)$ behaves like 
$$\left(\varepsilon h'(Y)\exp\Big\{X u_e^0 (0, Y)/\varepsilon \Big\},~ h(Y)\exp\Big\{ X u_e^0 (0, Y)/\varepsilon\Big\}\right)$$
 near $X=0$, where $h(Y)$ is a smooth compact support function in $Y \in (-1,1)$. Detailed constructions of $(u_p,v_p)$ and $(u_r, v_r)$ will be provided in Section \ref{sec_app}. Finally, $(u,v)$ represents the remainder. It is important to emphasize that the magnitude of $v_r$ is of order 1, while the magnitude of $v_p$ is of order $\sqrt{\varepsilon}$. Therefore, $v_r$ dominates in the vertical velocity.  We would like to point out that a similar profile to $(u_r, v_r)$ was obtained by Temam and Wang \cite{TemamWang} where they studied the boundary layer induced by suction. While they imposed periodic boundary conditions in the other variable, resulting in the absence of a Prandtl boundary layer, our result demonstrates that the layer profile $(u_r, v_r)$ and the Prandtl boundary layer can coexist.

In this article, we assume that our background Euler flow $(u_e^0, v_e^0)$ satisfies the following assumptions:
\begin{align}
&0 < c_0 \le u_e^0 \le C_0 < \infty,\label{assumption1}\\
&\left\|\frac{v_e^0}{\widetilde{Y}}\right\|_\infty \ll1,\label{assumption2}\\
&\| \nabla^m u_e^0 \|_\infty < \infty \quad \mbox{for}~ m \ge 0, \label{assumption3}\\
&\| \nabla^m v_e^0 \|_\infty < \infty \quad \mbox{for}~ m \ge 0, \label{assumption4} 
\end{align}
where $\widetilde{Y}:= (1-Y)(1+Y)$. The existence of such an Euler flow can be obtained through a perturbation from a shear flow (see \cite{Iyer19}*{Proposition 24}, for example).

Now we state our main theorems. The space $\cX$ we use to control the remainder is defined in \eqref{def_spaceX}. The first result is the existence of solution for small $L$.
\begin{theorem}\label{Thm1}
Assume that the background Euler flow $(u_e^0, v_e^0)$ satisfies \eqref{assumption1}-\eqref{assumption4} and $L\ll 1$, let $(u_p,v_p)$ be the alternating Prandtl-Euler corrector as in \eqref{Prandtl corrector}, $h(Y) \in C_c^\infty(-1,1)$ satisfying $\|h\|_\infty \ll 1$,  and $(u_r,v_r)$ be the profile constructed in \eqref{def_u_r v_r}. Then for any $0 < \varepsilon \ll L$, the equation \eqref{NS} admits a solution $(u^\varepsilon, v^\varepsilon)$ in the form of \eqref{expansion}, with the boundary conditions matching the approximated profile:
\begin{align*}
&(u^\varepsilon, v^\varepsilon)|_{Y=\pm 1} = 0,\\
&(u^\varepsilon, v^\varepsilon)|_{X=-L} = (u_e^0 + u_p, v_e^0 + v_p)|_{X=-L},\\
&(u^\varepsilon, v^\varepsilon)|_{X=0} = (u_e^0 + u_p + \varepsilon h(Y)/u_e^0 + O(\varepsilon^{3/2}) , v_e + v_p + h(Y) + O(\varepsilon^{1/2}))|_{X=0},
\end{align*}
and the remainder $(u,v)\in \cX$ with the estimate
$$
\|\{u,v\}\|_{\cX} \le 1.
$$
\end{theorem}

For the second result, we restrict the background Euler flow to be a shear, and make an additional concavity assumption on the Prandtl profile constructed, we can obtain the same existence result for large $L$.
\begin{theorem}\label{Thm2}
Assume that the background Euler flow is a shear flow $(u_e^0(Y), 0)$ that satisfies \eqref{assumption1} and \eqref{assumption3}, let $L > 0$, $(u_p,v_p)$ be the alternating Prandtl-Euler corrector as in \eqref{Prandtl corrector} with
\begin{equation}\label{concave}
- u^0_{pYY} \ge 0,
\end{equation}
$h(Y) \in C_c^\infty(-1,1)$ satisfying $\|h\|_\infty \ll 1$,  and $(u_r,v_r)$ be the profile constructed in \eqref{def_u_r v_r}. Then for any $0 < \varepsilon \ll L$, the equation \eqref{NS} admits a solution $(u^\varepsilon, v^\varepsilon)$ in the form of \eqref{expansion}, with the boundary conditions matching the approximated profile:
\begin{align*}
&(u^\varepsilon, v^\varepsilon)|_{Y=\pm 1} = 0,\\
&(u^\varepsilon, v^\varepsilon)|_{X=-L} = (u_e^0 + u_p, v_e^0 + v_p)|_{X=-L},\\
&(u^\varepsilon, v^\varepsilon)|_{X=0} = (u_e^0 + u_p + \varepsilon h(Y)/u_e^0 + O(\varepsilon^{3/2}) , v_e + v_p + h(Y) + O(\varepsilon^{1/2}))|_{X=0},
\end{align*}
and the remainder $(u,v)\in \cX$ with the estimate
$$
\|\{u,v\}\|_{\cX} \le 1.
$$
\end{theorem}

\begin{remark}
It is important to note that, unlike the small scale in the $Y$ variable found in the classical Prandtl boundary layer, where the small scale needs to be introduced through the boundary conditions at $X = -L$, the small scale from $(u_r,v_r)$ cannot be seen on the boundary. In this sense, the small scale from $(u_r,v_r)$ is more spontaneous.
\end{remark}

\begin{remark} The assumption \eqref{concave} is natural in the sense that an important class of self-similar solution to \eqref{Prandtl}, the Blasius solutions, satisfy \eqref{concave}. It is also known that a general class of solution to \eqref{Prandtl} converges to the Blasius solutions as $x \to \infty$ (see, e.g., \cites{Serrin67,Iyer20}).
\end{remark}

Proofs of these two theorems consist of two steps. We first construct an approximated profile
$$
(u_s, v_s) := (u_a + u_r, v_a + v_r):=(u_e^0 + u_p + u_r, v_e^0 + v_p + v_r).
$$
For convenience, we often work with the stream function $\phi(X,Y)$, where
$$
u = \phi_{Y}, \quad v = -\phi_{X}.
$$
Similarly, we use $\phi_s, \phi_a, \phi_e, \phi_p, \phi_r,$ to denote the stream function of $u_s, u_a, u_e^0, u_p, u_r$, respectively. 
A crucial quantity, the ``quotient", is defined as
$$
q := \frac{\phi}{u_a}.
$$
This quantity plays an important role in establishing the validity of steady Prandtl layer expansions in \cite{GuoIyer23}.

After constructing the approximated profile, the remainder $\phi$ satisfies
\begin{equation}\label{remainder_eq}
u_s \Delta \phi_X - \Delta u_s \phi_X + v_s \Delta \phi_Y - \Delta v_s \phi_Y - \varepsilon \Delta^2 \phi = -\varepsilon^{-\frac{13}{2}} R - \varepsilon^{\frac{13}{2}} (\phi_Y \Delta \phi_X - \phi_X \Delta \phi_Y),
\end{equation}
where 
\begin{equation}\label{remainder_R}
R = u_s \Delta \phi_{sX} + v_s \Delta \phi_{sY} - \varepsilon \Delta^2 \phi_s.
\end{equation}
The space $\cX$ is defined via
\begin{equation}\label{def_spaceX}
\begin{aligned}
\cX := \{ \phi \in H^3(\Omega): \phi = \partial_\nu \phi = 0~\mbox{on}~\partial \Omega, \| \phi\|_{\cX} < \infty\},
\end{aligned}
\end{equation}
where
$$
\|\phi\|_{\cX} := \| \nabla \phi\|_2 + \varepsilon^{1/2}\|\nabla^2 \phi\|_2 + \varepsilon^{1/2}\|\sqrt{u_a} \nabla^2 q\|_2 + \varepsilon^3 \| \nabla^3 \phi\|_2 < \infty.
$$

Since the right-hand side of \eqref{remainder_eq} is small, the second step is to establish the linear stability of approximated profile $\phi_s$. We aim to derive an estimate for the remainder $\phi$ in the space $\cX$, which contains an $\varepsilon$-independent $H^1$ estimate. This estimate arises from the term $u_s \Delta \phi_X$, whose coefficient $u_s = 0$ at $Y = \pm 1$. On the other hand, the profile $v_r$ introduces a highly singular coefficient:
$$\Delta v_r \sim \frac{1}{\varepsilon^2} h(Y) e^{Xu_e^0(0,Y)/\varepsilon}.$$
Therefore, equation \eqref{remainder_eq} can be interpreted as a mixed degenerate-singular elliptic equation.

To obtain the desired estimate for $\phi$, we follow a similar approach as in \cite{GaoZhang23}. Due to the presence of the singular coefficient $\Delta v_r$, we utilize a different Hardy-type inequality \eqref{Hardy_2}. To close the estimate, we make use of the weight $X$,which has been effectively used in \cites{Iyer19, IyerMasmoudi21a, GaoZhang23}. It is important to note that the weight $X$ interacts favorably with the profile $v_r$, as $X\Delta v_r \sim \varepsilon^{-1}$ reduces the singularity by a factor of $\varepsilon$.

The rest of this article is planned as follows. In Section \ref{sec_app}, we provide the construction of the approximated profile $\phi_s$. Estimates for \eqref{remainder_eq} will be given in Section \ref{sec_est}. Finally in Section \ref{sec_proof}, we prove Theorems \ref{Thm1} and \ref{Thm2}.

\section{Construction of the approximated profiles}\label{sec_app}

In this section, we outline the construction of the approximated profile $\phi_p$ and $\phi_r$, given a background Euler flow $\phi_e$. The profile $\phi_p$ is composed of alternating Prandtl and Euler correctors, a well-established approach. Therefore, we will only provide an overview of the strategy for constructing $\phi_p$ and the corresponding estimates. For a detailed construction of $\phi_p$,  readers are referred to \cites{Oleinik, GuoIyer21, GuoIyer23, IyerMasmoudi21b}. The main novelty is the construction of $\phi_r$, given the profile $\phi_e + \phi_p$.

\subsection{Construction of $\phi_p$}
We start with the asymptotic expansion:
\begin{equation}\label{Prandtl corrector}
\begin{aligned}
u_p &= u_{p}^0 + \sum_{i=1}^{15} \sqrt{\varepsilon}^i(u_e^i + u_p^i),\\
v_p &= \sum_{i=1}^{15} \sqrt{\varepsilon}^i( v_p^{i-1}+v_e^i) + \sqrt{\varepsilon}^{16} v_p^{15},\\
p_p &= p_p^0 + \sum_{i=1}^{15} \sqrt{\varepsilon}^i(p_e^i + p_p^i).
\end{aligned}
\end{equation}
Each $(u_p^i,v_p^i)$ consists of two component: one supported near $Y=-1$, denoted by $(u_p^{i,-},v_p^{i,-})$; the other one supported near $Y=1$, denoted by $(u_p^{i,+},v_p^{i,+})$. The role of the first Prandtl corrector $(u_p^0,v_p^0)$ is to balance the positive boundary value $u_e$ at $Y = \pm 1$. This process generates an error of size $\sqrt{\varepsilon}$ and introduces a nonzero boundary value of $v_p^0$ at $Y = \pm 1$, which will be subsequently corrected by the Euler corrector $(u_e^1, v_e^1)$. And this process will repeat recurrently.

For the Prandtl correctors, we will only describe the construction and the properties of $(u_p^{i,-},v_p^{i,-})$. Define the scaled variables
\begin{equation}\label{Prandtl variables}
x = X, \quad y = \frac{Y+1}{\sqrt{\varepsilon}}.
\end{equation}

Then $(u_p^{0,-},v_p^{0,-})$ satisfies the nonlinear Prandtl equation \eqref{Prandtl} with the boundary conditions
$$
u_p^{0,-}|_{x = -L} = U^0_p, \quad u_p^{0,-}|_{y = 0} = -u_e^0|_{Y = -1}, \quad u_p^{0,-}|_{y \to \infty} = v_p^{0,-}|_{y \to \infty} = 0.
$$
Under some compatibility conditions of the boundary data $U_p^0$, the existence of solution was established by Oleinik \cite{Oleinik}, and the higher order regularties were obtained by Guo and Iyer \cite{GuoIyer21} for small $x$, and by Wang and Zhang \cite{WangZhang21} for large $x$.

\begin{theorem}
Assume the prescribed boundary data $U_p^0(y) \in C^\infty$ satisfies
$$
U_p^0 >0 ~~ \mbox{for}~ y>0, \quad \partial_y U_p^0(0) > 0, \quad \partial_y^2 U_p^0(y) - p^0_{px}(0) \sim y^2~~\mbox{near}~y=0.
$$
Then for some $L>0$ ($L$ can be any positive number if $p_{px}^0 \le 0$), there exists a solution $(u_p^{0,-}, v_p^{0,-})$ to \eqref{Prandtl} satisfying, for some $y_0,m_0 > 0$,
$$
\sup_{x\in (-L,0)} \sup_{y \in (0,y_0)} |u_p^{0,-}, v_p^{0,-}, \partial_y u_p^{0,-}, \partial_{yy} u_p^{0,-}, \partial_x u_p^{0,-}| \lesssim 1,
$$
\begin{equation}\label{Prandtl_growth}
\sup_{x\in (-L,0)} \sup_{y \in (0,y_0)} \partial_y u_p^0 > m_0 > 0.
\end{equation}
Further more, fix an $N > 0$, assume generic compatibility conditions (see \cite{GuoIyer21}*{Definition 1}) at the corner $(-L,0)$, and the exponential decay of derivatives:
$$
\|\partial_y^j U_p^0 e^{Ny} \|_{L^\infty} \le C_0 \quad \mbox{for}~j \ge 0.
$$
Then
\begin{equation*}
\|e^{Ny}\nabla^j\{ u_p^{0,-},v_p^{0,-}\}  \|_{L^\infty} \lesssim C_0 \quad \mbox{for}~j \ge 0.
\end{equation*}
\end{theorem}

The theorem above provides the existence of $(u_p^{0,-}, v_p^{0,-})$, and $(u_p^{0,+}, v_p^{0,+})$ can be constructed in a similar manner. Since we require that $(u_p^{0,-} = v_p^{0,-}) = 0$ at $Y=1$, we introduce a smooth cut-off function $\chi(Y) \in C_c^\infty([0,2))$ such that $\chi = 1$ in $[0,1/2]$. We then define the first Prandtl corrector $(u_p^0, v_p^0)$ as follows:
\begin{equation}\label{u_p^0}
\begin{aligned}
u_p^0(X,Y):= &\chi(Y+1)u_p^{0,-}(X,\frac{Y+1}{\sqrt{\varepsilon}}) - \sqrt{\varepsilon}\chi'(Y+1) \int_0^X v_p^{0,-} (s, \frac{Y+1}{\sqrt{\varepsilon}}) \, ds\\
+& \chi(1-Y)u_p^{0,+}(X,\frac{1-Y}{\sqrt{\varepsilon}}) + \sqrt{\varepsilon}\chi'(1-Y) \int_0^X v_p^{0,-} (s, \frac{1-Y}{\sqrt{\varepsilon}}) \, ds,\\
v_p^0(X,Y):= & \chi(Y+1) v_p^{0,-} (X, \frac{Y+1}{\sqrt{\varepsilon}}) + \chi(1-Y) v_p^{0,-} (X, \frac{1-Y}{\sqrt{\varepsilon}}).
\end{aligned}
\end{equation}
Plugging into the equation \eqref{Prandtl}, we observe that the error introduced by the cut-off function is of order $\sqrt\varepsilon$, which contributes to the forcing term for the next Prandtl layer.

The $i$-th Euler corrector $(u_e^i, v_e^i)$ satisfies the linearized Euler equation around the background Euler flow $(u_e^0, v_e^0)$:
\begin{equation}\label{i-th Euler}
\left\{
\begin{aligned}
&u_e^0 u^i_{eX} + u_{eX}^0 u_e^i +  v_e^0 u^i_{eY} + u_{eY}^0 v_e^i + p_{eX}^i = f_{e,1}^i,\\
&u_e^0 v^i_{eX} + v_{eX}^0 u_e^i +  v_e^0 v^i_{eY} + v_{eY}^0 v_e^i + p_{eY}^i = f_{e,2}^i,\\
&u^0_X + v^0_Y = 0,\\
&v_e^i|_{Y = \pm 1} = - u_p^{i-1} |_{Y = \pm 1}, \quad v_e^i |_{X = -L, 0} = V^i_{e, \{-L,0\}}, \quad u_e^i |_{X = -L} = U_e^i,
\end{aligned}
\right.
\end{equation}
where the forcing terms $f_{e,1}^i$ and $f_{e,2}^i$ are errors produced from the previous correctors. Assuming smoothness for the boundary datum and the compatibility conditions at the corners, by the standard elliptic theory, there exists a solution $(u_e^i, v_e^i)$ to \eqref{i-th Euler} satisfying
$$
\|\{u_e^i, v_e^i\} \|_{H^k} \lesssim 1, \quad k \ge 0.
$$
The $i$-th Prandtl corrector for the lower boundary $(u_p^{i,-},v_p^{i,-})$ satisfies the linearized Prandtl equation around $(\bar u, \bar v): = \Big(u_p^0+ u_e^0(X,-1), v_p^0- v_p^0(X,-1)\Big)$:
\begin{equation}\label{i-th Prandtl}
\left\{
\begin{aligned}
&\bar u u^{i,-}_{px} + u^{i,-}_{p} \bar{u}_x +  \bar{u}_{y}[v_p^{i,-} - v_p^{i,-}|_{y=0}] + \bar{v} u_{py}^{i,-} + p_{px}^{i,-} - u_{pyy}^{i,-} = f^i,\\
&p_{py}^{i,-} = 0,\\
&u_{px}^{i,-} + v_{py}^{i,-} = 0,\\
& u_p^{i,-}|_{y = 0} = -u_e^i|_{Y = -1}, \quad u_p^{i,-}|_{y \to \infty} = v_p^{i,-}|_{y \to \infty} = 0, \quad v_p^{i,-}|_{x = -L} = V^i_p,
\end{aligned}
\right.
\end{equation}
where $(x,y)$ are the Prandtl variables \eqref{Prandtl variables}, the forcing term $f^i$ is errors produced from the previous corrector. Note that for the final layer, we will impose $v_p^{15,-}|_{y =0} = 0$ instead of $v_p^{15,-}|_{y \to \infty} = 0$ and perform a cut-off, whose error is negligible. The existence of solution to \eqref{i-th Prandtl} and the corresponding estimates were established in \cite{GuoIyer21}.

\begin{theorem}
Assume the prescribed boundary data $V_p^i(y) \in C^\infty$ satisfies generic compatibility conditions at the corner $(-L,0)$, and the exponential decay of derivatives for some $N>0$:
$$
\|\partial_y^j V_p^i e^{Ny} \|_{L^2} \le C_0 \quad \mbox{for}~j \ge 0.
$$
Assume the forcing $f^i$ satisfies the bounds
$$
\|\partial_x^k \partial_y f^i e^{Ny} \|_{L^2}  \le C_1 \quad \mbox{for}~k \ge 0.
$$
there exists a solution $(u_p^{i,-}, v_p^{i,-})$ to \eqref{Prandtl} satisfying
$$
\|e^{Ny}\nabla^{\alpha}\{ u_p^{i,-},v_p^{i,-}\}  \|_{L^\infty} \le C(C_0,C_1) \quad \mbox{for}~\alpha \ge 0.
$$
\end{theorem}
After constructing $(u_p^{i,\pm}, v_p^{i,\pm})$, $(u_p^i, v_p^i)$ will be constructed similarly to \eqref{u_p^0}. This completes the construction of $\phi_p$. Recall the notations
$$\phi_a = \phi_e + \phi_p, \quad u_a = \phi_{aY}, \quad \mbox{and} \quad v_a = - \phi_{aX}.$$ 
Through a straightforward computation, we obtain
\begin{equation}\label{phi_a_error}
u_a \Delta \phi_{aX} + v_a \Delta \phi_{aY} - \varepsilon \Delta^2 \phi_a= O(\varepsilon^{7}).
\end{equation}

\subsection{Construction of $\phi_r$}

We now turn to the construction of $\phi_r$. The objective is to construct a $\phi_r$ with small scales in $X$ so that the error
$$
R = (u_a + u_r) \Delta (\phi_{aX} + \phi_{rX}) + (v_a + v_r) \Delta (\phi_{aY} + \phi_{rY}) - \varepsilon \Delta^2 (\phi_a + \phi_r) = O(\varepsilon^{7}).
$$
We introduce the scaled variables
$$
x = \frac{X}{\varepsilon}, \quad y = Y,
$$
which will be used throughout this subsection. First, we formally write
$$
\phi_r(X,Y) = \sum_{i=0}^{17} \varepsilon^{1+i/2} \phi^i (x,y),
$$
and plug it into the error term, we have
\begin{align*}
R =& \Big(u_a + \varepsilon\sum_{i=0}^{17}\varepsilon^{i/2} \phi^i_y \Big) \Big(-\Delta v_a + \varepsilon^{-2} \sum_{i=0}^{17}\varepsilon^{i/2} \phi^i_{xxx} + \sum_{i=0}^{17}\varepsilon^{i/2} \phi^i_{xyy} \Big)\\
&+ \Big(v_a - \sum_{i=0}^{17}\varepsilon^{i/2} \phi^i_x \Big) \Big(\Delta u_a + \varepsilon^{-1} \sum_{i=0}^{17}\varepsilon^{i/2} \phi^i_{xxy} + \varepsilon\sum_{i=0}^{17}\varepsilon^{i/2} \phi^i_{yyy} \Big)\\
&- \varepsilon \Delta^2 \phi_a - \varepsilon^{-2} \sum_{i=0}^{17}\varepsilon^{i/2} \phi^i_{xxxx} - \varepsilon^2 \sum_{i=0}^{17}\varepsilon^{i/2} \phi^i_{yyyy} - 2 \sum_{i=0}^{17}\varepsilon^{i/2} \phi^i_{xxyy}.
\end{align*}
We will construct the profiles $\phi^i$ to be compactly supported in $y$. Recall that $(u_p^j , v_p^j)$ is the $j$-th Prandtl layer constructed in \eqref{Prandtl corrector}, which has fast decay away from $Y = \pm 1$. Therefore,
$$
|u_p^j \phi^i| + |v_p^j \phi^i| = O(\varepsilon^{10})
$$
for any $i,j$ when $\varepsilon$ is small. Combining with \eqref{phi_a_error}, we can rewrite
\begin{equation}\label{R}
\begin{aligned}
R =& \Big(\sum_{i=0}^{15} \varepsilon^{i/2} u_e^i + \varepsilon\sum_{i=0}^{17}\varepsilon^{i/2} \phi^i_y \Big) \Big(\varepsilon^{-2} \sum_{i=0}^{17}\varepsilon^{i/2} \phi^i_{xxx} + \sum_{i=0}^{17}\varepsilon^{i/2} \phi^i_{xyy} \Big)\\
&- \Big(\varepsilon\sum_{i=0}^{17}\varepsilon^{i/2} \phi^i_y \Big)\Big( \sum_{i=0}^{15} \varepsilon^{i/2} \Delta v_e^i \Big) - \Big( \sum_{i=0}^{17}\varepsilon^{i/2} \phi^i_x \Big) \Big(\sum_{i=0}^{15} \varepsilon^{i/2} \Delta u_e^i \Big)\\
&+ \Big(\sum_{i=0}^{15} \varepsilon^{i/2} v_e^i - \sum_{i=0}^{17}\varepsilon^{i/2} \phi^i_x \Big) \Big( \varepsilon^{-1} \sum_{i=0}^{17}\varepsilon^{i/2} \phi^i_{xxy} + \varepsilon\sum_{i=0}^{17}\varepsilon^{i/2} \phi^i_{yyy} \Big)\\
&- \varepsilon^{-2} \sum_{i=0}^{17}\varepsilon^{i/2} \phi^i_{xxxx} - \varepsilon^2 \sum_{i=0}^{17}\varepsilon^{i/2} \phi^i_{yyyy} - 2 \sum_{i=0}^{17}\varepsilon^{i/2} \phi^i_{xxyy} + O(\varepsilon^{7}).
\end{aligned}
\end{equation}
The most singular term is $\varepsilon^{-2} (u_e^0 \phi^0_{xxx} - \phi^0_{xxxx})$. Let $h(Y)$ be a smooth, compactly supported function, which will serve as the leading-order boundary value of $v_r$ at $X = 0$. We define
$$
\phi^0(x,y) = \frac{h(y)}{u^0_e(0,y)} \exp\Big\{\int_0^{x} u_e^0 (\varepsilon x', y) \, dx'\Big\},
$$
then
$$
u_e^0 \phi^0_{xxx} - \phi^0_{xxxx} = 0,
$$
and $(\varepsilon\phi^0)_X |_{X = 0} = h(Y)$. It is straightforward to see that
\begin{equation}\label{phi^0 decay}
|\nabla^m \phi^0(x,y)| \lesssim \exp\Big\{\int_0^{x} u_e^0 (\varepsilon x', y) \, dx'\Big\},\quad (x,y) \in (-L/\varepsilon, 0)\times(-1,1), m\ge 0.
\end{equation}
Next, we define $\phi^k$ inductively to match the terms of order $\varepsilon^{-2 + k/2}$ in \eqref{R}. By collecting the $O(\varepsilon^{-2 + k/2})$ terms in \eqref{R} for $1 \le k \le 17$, we have
\begin{equation}\label{phi^k}
u_e^0\phi_{xxx}^k - \phi_{xxxx}^k = f_k,
\end{equation}
with
\begin{align*}
f_k(x,y) =& - \sum_{i=1}^k u_e^i \phi_{xxx}^{k-i} - \sum_{i=0}^{k-4} u_e^i \phi_{xyy}^{k-4-i} - \sum_{i=0}^{k-2} \phi_y^i \phi_{xxx}^{k-2-i} - \sum_{i=0}^{k-6} \phi_y^i \phi_{xyy}^{k-6-i}\\
&+ \sum_{i=0}^{k-6} \phi_y^i \Delta v_e^{k-6-i} + \sum_{i=0}^{k-4} \phi_x^i \Delta u_e^{k-4-i} - \sum_{i=0}^{k-2} v_e^i \phi_{xxy}^{k-2-i} - \sum_{i=0}^{k-6}v_e^i \phi_{yyy}^{k-6-i}\\ 
&+ \sum_{i=0}^{k-2} \phi_x^i \phi_{xxy}^{k-2-i} + \sum_{i=0}^{k-6}\phi_x^i \phi_{yyy}^{k-6-i} + \phi_{yyyy}^{k-8} + 2 \phi_{xxyy}^{k-4},
\end{align*}
where we adopt the convention that a summation equals zero when the upper limit is strictly less than its lower limit.We also extend the definitions of $u_e^j$ and $v_e^j$ to be 0 when $j > 15$, and $\phi^j$ to be 0 when $j < 0$. We will solve \eqref{phi^k} inductively with the boundary conditions
\begin{equation*}\label{phi^k boundary conditions}
\phi^k (-L/\varepsilon, y) = \phi^k_{x} (-L/\varepsilon, y) = \phi^k_{xx} (-L/\varepsilon, y) = \phi^k_{xxx} (0, y) = 0.
\end{equation*}
When $k = 1$, $f_1(x,y) = -u_e^1(\varepsilon x, y) \phi_{xxx}^0 (x,y)$. We can explicitly solve $\phi^1$ as
\begin{align*}
&\phi^1(x,y) =\\
&\int_{-\frac{L}{\varepsilon}}^x \int_{-\frac{L}{\varepsilon}}^{x'''''} \int_{-\frac{L}{\varepsilon}}^{x''''} e^{\int_0^{x'''} u_e^0 (\varepsilon x', y) \, dx'} \int_{0}^{x'''} e^{-\int_0^{x''} u_e^0 (\varepsilon x', y) \, dx'} f_1(x'',y) \, dx'' dx''' dx'''' dx'''''.
\end{align*}
From \eqref{phi^0 decay}, we have
\begin{align}\label{phi_1}
|\phi^1(x,y)| \lesssim \int_{-\frac{L}{\varepsilon}}^x \int_{-\frac{L}{\varepsilon}}^{x'''''} \int_{-\frac{L}{\varepsilon}}^{x''''} -x''' e^{\int_0^{x'''} u_e^0 (\varepsilon x', y) \, dx'} dx''' dx'''' dx'''''.
\end{align}
By \eqref{assumption1} and integration by parts, we can estimate the first integral
\begin{equation*}
\begin{aligned}
&\int_{-\frac{L}{\varepsilon}}^{x''''} -x''' e^{\int_0^{x'''} u_e^0 (\varepsilon x', y) \, dx'} dx''' \lesssim \int_{-\frac{L}{\varepsilon}}^{x''''} -x'''u_e^0 (\varepsilon x''', y) e^{\int_0^{x'''}  u_e^0 (\varepsilon x', y) \, dx'} dx''' \\
=& \int_{-\frac{L}{\varepsilon}}^{x''''} -x''' \, d\Big(e^{\int_0^{x'''} u_e^0 (\varepsilon x', y) \, dx'} \Big)\\
=& -x'''' e^{\int_0^{x''''} u_e^0 (\varepsilon x', y) \, dx'} - \frac{L}{\varepsilon} e^{\int_0^{-\frac{L}{\varepsilon}} u_e^0 (\varepsilon x', y) \, dx'} + \int_{-\frac{L}{\varepsilon}}^{x''''} e^{\int_0^{x'''} u_e^0 (\varepsilon x', y) \, dx'} \, dx'''\\
\lesssim&  (1-x'''') \exp\Big\{\int_0^{x''''} u_e^0 (\varepsilon x', y) \, dx'\Big\},  \quad (x'''',y) \in (-L/\varepsilon, 0)\times(-1,1).
\end{aligned}
\end{equation*}
Similarly, one can apply the similar argument for the other two integrals in \eqref{phi_1}, and obtain for $m \ge 0$,
$$
|\nabla^m \phi^1(x,y)| \lesssim (1-x) \exp\Big\{\int_0^{x} u_e^0 (\varepsilon x', y) \, dx'\Big\},  \quad (x,y) \in (-L/\varepsilon, 0)\times(-1,1).
$$
Next, assuming that 
$$
|\nabla^m f_k(x,y)| \lesssim (1-x)^j \exp\Big\{\int_0^{x} u_e^0 (\varepsilon x', y) \, dx'\Big\}
$$
for some $j \ge 0$, we define
\begin{align*}
&\phi^k(x,y) :=\\
&\int_{-\frac{L}{\varepsilon}}^x \int_{-\frac{L}{\varepsilon}}^{x'''''} \int_{-\frac{L}{\varepsilon}}^{x''''} e^{\int_0^{x'''} u_e^0 (\varepsilon x', y) \, dx'} \int_{0}^{x'''} e^{-\int_0^{x''} u_e^0 (\varepsilon x', y) \, dx'} f_k(x'',y) \, dx'' dx''' dx'''' dx'''''.
\end{align*}
We then obtain the estimate
$$
|\nabla^m \phi^k(x,y)| \lesssim (1-x)^{j+1} \exp\Big\{\int_0^{x} u_e^0 (\varepsilon x', y) \, dx'\Big\},  \quad (x,y) \in (-L/\varepsilon, 0)\times(-1,1)
$$
similarly as above. This process allows us to define $\phi^k$ inductively, so that it satisfies 
\begin{equation}\label{phi^k decay}
|\nabla^m \phi^k(x,y)| \le C_1 e^{C_2x},  \quad (x,y) \in (-L/\varepsilon, 0)\times(-1,1), m \ge 0, 0 \le k \le 17
\end{equation}
for some positive constants $C_1$ and $C_2$ depending on $m$. Finally, let $\chi \in C_c^\infty([0,1))$ be a smooth cut-off function such that $\chi = 1$ on $[0,1/2)$, we define
$$
\phi_r(X,Y) = \sum_{i=0}^{17} \varepsilon^{1+i/2} \chi(-\frac{X}{L}) \phi^i (\frac{X}{\varepsilon},Y),
$$
and
\begin{equation}\label{def_u_r v_r}
u_r(X,Y) = - \phi_{rY} (X,Y), \quad v_r(X,Y) = \phi_{rX} (X,Y).
\end{equation}
Note that the error introduced by the cut-off function is exponentially small in $\varepsilon$ due to \eqref{phi^k decay}, and the profile $(u_r,v_r)$ satisfies the boundary conditions:
\begin{align*}
&(u_r,v_r)|_{Y = \pm 1} = (u_r,v_r)|_{X = - L} = 0, \\
&(u_r,v_r)|_{X=0} = (\varepsilon h(Y)/u_e^0 + O(\varepsilon^{3/2}), h(Y) + O(\varepsilon^{1/2})).
\end{align*}
More importantly, near $X = 0$,
$$
v_r \sim h(Y) \frac{u_e^0(X,Y)}{u_e^0(0,Y)} \exp\Big\{\int_0^{\frac{X}{\varepsilon}} u_e^0 (\varepsilon x', Y) \, dx'\Big\},
$$
which reveals a boundary layer of thickness $\varepsilon$ near $X = 0$.

\section{Estimates for the system of remainders}\label{sec_est}

In this section, we study the linear version of \eqref{remainder_eq}:
\begin{equation}\label{remainder_linear}
u_s \Delta \phi_X - \Delta u_s \phi_X + v_s \Delta \phi_Y - \Delta v_s \phi_Y - \varepsilon \Delta^2 \phi = f
\end{equation}
with the boundary conditions
\begin{equation}\label{remainder_boundary_condition}
\begin{aligned}
&\phi|_{X=-L} = \phi|_{X=0} = \phi|_{Y=-1} = \phi|_{Y=1} \\
=& \phi_X|_{X=-L} = \phi_X|_{X=0} = \phi_Y|_{Y=-1} = \phi_Y|_{Y=1} = 0.
\end{aligned}
\end{equation}
We then establish an $H^2$ estimate. We recall the notations
$$
u_a= u_e^0 + u_p, \quad v_a= v_e^0 + v_p, \quad \mbox{and} \quad q= \frac{\phi}{u_a}.
$$
The equation \eqref{remainder_linear} then becomes
\begin{equation}\label{linear_quotient}
(u_a^2q_X)_{XX} + (u_a^2 q_Y)_{XY} - \varepsilon \Delta^2 \phi + J[\phi] = f,
\end{equation}
where
$$
J[\phi] = u_r \Delta \phi_X - \Delta u_r \phi_X - u_{aX} \Delta \phi + \Delta u_{aX} \phi + v_s \Delta \phi_Y - \Delta v_s \phi_Y,
$$
with the boundary conditions
\begin{align*}
q|_{X=-L} = q|_{X=0} = q|_{Y=-1} = q|_{Y=1} = q_X|_{X=-L} = q_X|_{X=0} = 0.
\end{align*}
Note that since $u_a = 0$ on $\{Y = \pm 1\}$, we do not have $q_Y = 0$ on $\{Y = \pm 1\}$. Recall that $\widetilde{Y}= (1-Y)(1+Y)$. First it is straightforward to see the following properties for the approximated profiles constructed in Section \ref{sec_app}.
\begin{lemma}
For $i,j \in \bZ$ such that $i + j \ge 0$, $m \in \bN$, we have
\begin{align}
\varepsilon^{i/2} \| \widetilde{Y}^j \partial_{Y}^{i+j} \partial_{X}^m u_p \|_\infty + \varepsilon^{(i-1)/2} \| \widetilde{Y}^j \partial_{Y}^{i+j} \partial_{X}^m v_p \|_\infty &\lesssim 1, \label{u_a_property}\\
\varepsilon^{i-1} \| X^j \partial_{X}^{i+j} u_r \|_\infty + \varepsilon^{i} \| X^j \partial_{X}^{i+j} v_r \|_\infty &\lesssim \|h\|_{\infty},\label{u_r_property}\\
\varepsilon^{i-1} \| X^j \partial_{X}^{i+j} \partial_{Y}^{m+1} u_r \|_\infty + \varepsilon^{i} \| X^j \partial_{X}^{i+j}  \partial_{Y}^{m+1} v_r \|_\infty &\lesssim 1.\label{u_r_property2}
\end{align}
\end{lemma}
We will use the following Hardy-type inequalities. Idea essentially follows from \cite{GuoIyer23}*{Lemma 3.7}.
\begin{lemma}
Let $F(X,Y) \in H^1$, $G(X,Y) \in H^1$ with $G(0,Y) = 0$, and $\sigma > 0$, we have
\begin{align}
&\|F\|_{2} \lesssim \sqrt{\varepsilon} \sqrt{\sigma} \|\sqrt{u_a} F_Y\|_2 + \frac{1}{\sigma} \|u_a F\|_2,\label{Hardy_1} \\
&\| \sqrt{u_a} \frac{G}{X} \|_2 \lesssim \| \sqrt{u_a} G_X \|_2. \label{Hardy_2}
\end{align}
\end{lemma}
\begin{proof}
First we prove \eqref{Hardy_1}. Let $\chi(Y)$ be a smooth cut-off function supported in $[0,1]$ and $\chi \equiv 1$ in $[0,1/2]$. We write
\begin{align*}
\|F\|_{2}^2 =& \|F \chi \left( \frac{Y+1}{\sigma\sqrt{\varepsilon}}\right) \|_{2}^2 + \|F \chi \left( \frac{-Y+1}{\sigma\sqrt{\varepsilon}}\right) \|_{2}^2\\
& + \|F \left[ 1 - \chi \left( \frac{Y+1}{\sigma\sqrt{\varepsilon}}\right) - \chi \left( \frac{-Y+1}{\sigma\sqrt{\varepsilon}}\right) \right] \|_{2}^2.
\end{align*}
For the first term, we estimate
\begin{align*}
\|F \chi \left( \frac{Y+1}{\sigma\sqrt{\varepsilon}}\right) \|_{2}^2 & = (\partial_Y\{Y+1\}, F^2 \chi^2 \left( \frac{Y+1}{\sigma\sqrt{\varepsilon}}\right) )\\
&= - 2((Y+1)F_Y, F \chi^2) - 2 (  \frac{Y+1}{\sigma\sqrt{\varepsilon}}, F^2 \chi' \chi )\\
& \lesssim \| (Y+1)F_Y \chi\|_2^2 + \| \left( \frac{Y+1}{\sigma\sqrt{\varepsilon}}\right) F \chi' \|_2^2\\
& \lesssim \varepsilon \sigma \|\sqrt{u_a} F_Y\|_2^2 + \frac{1}{\sigma^2} \|u_a F\|_2^2
\end{align*}
since $u_a \sim \frac{Y+1}{\sqrt{\varepsilon}}$ in the support of $\chi \left( \frac{Y+1}{\sigma\sqrt{\varepsilon}}\right)$. The second term $\|F \chi \left( \frac{-Y+1}{\sigma\sqrt{\varepsilon}}\right) \|_{2}^2$ can be estimated in the same way. For the third term, since $u_{aY} \gtrsim \varepsilon^{-1/2}$ when $Y < -1 + \sigma \sqrt{\varepsilon}$, and $u_{aY} \lesssim -\varepsilon^{-1/2}$ when $Y > 1 - \sigma \sqrt{\varepsilon}$, we have
$$
u_a \gtrsim \sigma \quad \mbox{when}~ Y \in (-1 + \frac{1}{2}\sigma \sqrt{\varepsilon} < Y < 1 - \frac{1}{2}\sigma \sqrt{\varepsilon} ).
$$
Therefore,
\begin{align*}
&\|F \left[ 1 - \chi \left( \frac{Y+1}{\sigma\sqrt{\varepsilon}}\right) - \chi \left( \frac{-Y+1}{\sigma\sqrt{\varepsilon}}\right) \right] \|_{2}^2\\ 
=& \|\frac{u_a}{u_a} F \left[ 1 - \chi \left( \frac{Y+1}{\sigma\sqrt{\varepsilon}}\right) - \chi \left( \frac{-Y+1}{\sigma\sqrt{\varepsilon}}\right) \right] \|_{2}^2\\
\lesssim& \frac{1}{\sigma^2} \|u_a F\|_2^2.
\end{align*}
For \eqref{Hardy_2}, we fix a function $\widetilde u_a$ that is a function of $Y$ only, and such that $\widetilde u_a (Y) \lesssim u_a(X,Y) \lesssim \widetilde u_a(Y)$ for all $-L < X < 0, -1 < Y < 1$. This is possible due to \eqref{Prandtl_growth}. Then by the Hardy's inequality in $X$ variable, we have
$$
\| \sqrt{u_a} \frac{G}{X} \|_2 \lesssim \| \sqrt{\widetilde u_a} \frac{G}{X} \|_2 \lesssim \| \sqrt{\widetilde u_a} G_X \|_2 \lesssim \| \sqrt{u_a} G_X \|_2.
$$
This concludes the proof.
\end{proof}
Next we prove a basic energy estimate.
\begin{lemma}\label{lem_1}
Let $q$ be a solution of \eqref{linear_quotient}. Assume that $\varepsilon \ll 1$ and $\|h\|_\infty \ll 1$, then we have
\begin{equation}\label{energy_1}
\varepsilon \left( \| \sqrt{u_a} q_{XX}\|_2^2 +\| \sqrt{u_a} q_{XY}\|_2^2 + \| \sqrt{u_a} q_{YY}\|_2^2 \right) \lesssim \| u_a q_{X}\|_2^2 + \| u_a q_{Y}\|_2^2 + \|f\|_2^2.
\end{equation}
\end{lemma}
\begin{proof}
We take the inner product of the equation \eqref{linear_quotient} with $-q$, and estimate each term. Since many of the terms are estimated in a similar manner, we will only point out the argument when it is first introduced.

\textbf{Step 1:} Rayleigh terms.

\begin{equation}\label{1-1}
\begin{aligned}
((u_a^2 q_{X})_{XX} , -q) &= ((u_a^2 q_{X})_{X} , q_X)\\
&= 2(u_a u_{aX} q_X, q_X) + (u_a^2 q_{XX}, q_X)\\
&= (u_a u_{aX} q_X, q_X)\\
&= O(1)\|q_X\|_2^2.
\end{aligned}
\end{equation}
Similarly, we have
\begin{equation}\label{1-2}
((u_a^2 q_{Y})_{XY} , -q) = ((u_a^2 q_{Y})_{X} , q_Y) = (u_a u_{aX} q_Y, q_Y) = O(1)\|q_Y\|_2^2.
\end{equation}

\textbf{Step 2:} $\Delta^2$ terms.

We split the Bi-Laplacian term into three parts
$$
(-\varepsilon\Delta^2 \phi, - q) = \varepsilon (\phi_{XXXX} + 2 \phi_{XXYY} + \phi_{YYYY}, q).
$$
For the first part,
\begin{align*}
\varepsilon (\phi_{XXXX}, q) &= \varepsilon (\phi_{XX}, q_{XX})\\
&=\varepsilon (u_{aXX}q + 2 u_{aX}q_X + u_a q_{XX}, q_{XX})\\
&=\varepsilon (u_a q_{XX}, q_{XX}) - 2\varepsilon(u_{aXX}q_X,q_X) - \varepsilon (u_{aXXX}q, q_X)\\
&=\varepsilon (u_a q_{XX}, q_{XX}) + O(\varepsilon)\|q_X\|_2^2,
\end{align*}
where we used the definition $\phi = u_a q$, \eqref{u_a_property} to bound $u_{aXX}$ and $u_{aXXX}$, and the Poincare inequality $\|q\|_2 \lesssim \|q_X\|_2$. For the second part,
\begin{align*}
2\varepsilon (\phi_{XXYY}, q) =& 2\varepsilon (\phi_{XY}, q_{XY})\\
=&2\varepsilon (u_{aXY}q +  u_{aX}q_Y + u_{aY} q_X + u_a q_{XY}, q_{XY})\\
=&2\varepsilon (u_a q_{XY}, q_{XY}) - \varepsilon(u_{aXX}q_Y,q_Y)- \varepsilon(u_{aYY}q_X,q_X) \\
&- 2\varepsilon (u_{aXYY}q, q_X)- 2\varepsilon (u_{aXY}q_Y, q_X)\\
=&2\varepsilon (u_a q_{XY}, q_{XY}) + O(1)\|q_X\|_2^2 + O(\varepsilon)\|q_Y\|_2^2,
\end{align*}
where we used \eqref{u_a_property} when $Y$ derivatives hit on $u_a$. For the last part, note that we do not have $q_Y = 0$ when $Y = \pm 1$,
\begin{equation}\label{1-2-1}
\begin{aligned}
\varepsilon (\phi_{YYYY}, q) =& -\varepsilon (\phi_{YYY}, q_{Y})\\
=& -\varepsilon (\phi_{YY}, q_{Y})|_{Y=1} + \varepsilon (\phi_{YY}, q_{Y})|_{Y=-1} + \varepsilon(\phi_{YY},q_{YY})\\
=& -2\varepsilon (u_{aY}q_{Y}, q_{Y})|_{Y=1} + 2\varepsilon (u_{aY}q_{Y}, q_{Y})|_{Y=-1}\\
&+\varepsilon (u_{aYY}q + 2 u_{aY}q_Y + u_a q_{YY}, q_{YY})\\
=&\varepsilon (u_a q_{YY}, q_{YY}) - 2\varepsilon(u_{aYY}q_Y,q_Y) - \varepsilon (u_{aYYY}q, q_Y)\\
&-\varepsilon (u_{aY}q_{Y}, q_{Y})|_{Y=1} + \varepsilon (u_{aY}q_{Y}, q_{Y})|_{Y=-1}
\end{aligned}
\end{equation}
Note that $u_{aY} > 0$ at $Y=-1$ and $u_{aY} < 0$ at $Y=1$, the last two terms $(\ref{1-2-1}.4)$ and $(\ref{1-2-1}.5)$ are nonnegative. $(\ref{1-2-1}.1)$ is a favorable term. To estimate the rest, we use \eqref{u_a_property} and the Hardy inequality $\|\frac{q}{\widetilde{Y}} \|_2 \lesssim \|q_Y\|_2$, we can get
\begin{align*}
&|(\ref{1-2-1}.2)| + |(\ref{1-2-1}.3)|\\
\lesssim& \|\varepsilon u_{aYY}\|_\infty\|q_Y\|_2^2 + \|\varepsilon \widetilde{Y} u_{aYYY}\|_{\infty} \|\frac{q}{\widetilde{Y}} \|_2 \|q_Y\|_2\\
=&  O(1)\|q_Y\|_2^2.
\end{align*}
Therefore,
\begin{equation}\label{1-3}
\begin{aligned}
(-\varepsilon\Delta^2 \phi, - q) =& \varepsilon (u_a q_{XX}, q_{XX}) +2\varepsilon (u_a q_{XY}, q_{XY}) + \varepsilon (u_a q_{YY}, q_{YY})\\
& -\varepsilon (u_{aY}q_{Y}, q_{Y})|_{Y=1} + \varepsilon (u_{aY}q_{Y}, q_{Y})|_{Y=-1}\\ 
&+O(1)\|q_X\|_2^2 + O(1)\|q_Y\|_2^2.
\end{aligned}
\end{equation}

\textbf{Step 3:} $J[\phi]$ terms.

Next, we estimate $(J[\phi],-q)$ term by term.
\begin{equation}\label{1-4}
\begin{aligned}
( u_r \Delta \phi_X, -q) =& (u_{rX} \phi_{XX}, q) + (u_{rY} \phi_{XY}, q) + (u_{r} \phi_{XX}, q_X) + (u_{r} \phi_{XY}, q_Y)\\
=&- (u_{rXX} \phi_{X}, q) - (u_{rX} \phi_{X}, q_X) - (u_{rXY} \phi_{Y}, q) - (u_{rY} \phi_{Y}, q_X)\\
&+ (u_{r} (u_{aXX}q + 2 u_{aX}q_X + u_a q_{XX}), q_X)\\
&+ (u_{r} (u_{aXY}q +  u_{aX}q_Y + u_{aY} q_X + u_a q_{XY}), q_X)\\
\lesssim& \|Xu_{rXX}\|_\infty \|\phi_X\|_2\|\frac{q}{X}\|_2 + \|u_{rX}\|_\infty\|\phi_X\|_2\|q_X\|_2\\
& + \|u_{rXY}\|_\infty\|\phi_Y\|_2\|q\|_2+ \|u_{rY}\|_\infty\|\phi_Y\|_2\|q_X\|_2\\
&+ \|u_ru_{aXX}\|_\infty \|q\|_2\|q_X\|_2 + \|u_ru_{aX}\|_\infty \|q_X\|_2^2\\
&+ \|(u_ru_{a})_X\|_\infty \|q_X\|_2^2 + \|u_ru_{aXY}\|_\infty \|q\|_2\|q_X\|_2\\
&+ \|u_ru_{aX}\|_\infty \|q_Y\|_2\|q_X\|_2+ \|u_ru_{aY}\|_\infty \|q_X\|_2^2 + \|(u_ru_{a})_Y\|_\infty \|q_X\|_2^2\\
=& O(1) \|q_X\|_2^2 + O(1)\|q_Y\|_2^2,
\end{aligned}
\end{equation}
where we used
\begin{equation}\label{phi_X to q_X}
\begin{aligned}
\|\phi_X\|_2 &\lesssim \|u_{aX} q\|_2+ \|u_a q_X\|_2 \lesssim \|q_X\|_2,\\
\|\phi_Y\|_2 &\lesssim \|u_{aY} q\|_2+ \|u_a q_Y\|_2 \lesssim \|\widetilde{Y} u_{aY} \|_\infty \|\frac{q}{\widetilde{Y}}\|_2 + \|u_a q_Y\|_2 \lesssim \|q_Y\|_2,
\end{aligned}
\end{equation}
and the fact that $u_r$ is compactly supported in $Y \in (-1,1)$, so its interaction with any contribution from the Prandtl corrector is small when $\varepsilon$ is small enough. More precisely, we used
\begin{align*}
&\|u_ru_{aXX}\|_\infty +  \|u_ru_{aX}\|_\infty +\|(u_ru_{a})_X\|_\infty\\
&+ \|u_ru_{aXY}\|_\infty + \|u_ru_{aY}\|_\infty + \|(u_ru_{a})_Y\|_\infty \lesssim 1.
\end{align*}
By the similar argument, we can estimate
\begin{equation}\label{1-5}
\begin{aligned}
(- \Delta u_r \phi_X, - q) =& (u_{rXX} \phi_X, q) + (u_{rYY} \phi_X, q)\\
\lesssim & \|Xu_{rXX}\|_\infty \|\phi_X\|_2\|\frac{q}{X}\|_2 +\|u_{rYY}\|_\infty \|\phi_X\|_2\|q\|_2\\
=& O(1)\|q_X\|_2^2,
\end{aligned}
\end{equation}

\begin{equation}\label{1-6}
\begin{aligned}
(- u_{aX} \Delta \phi , -q) =& (u_{aX} \phi_{XX}, q) + (u_{aX} \phi_{YY}, q)\\
=& -(u_{aXX} \phi_{X}, q) - (u_{aX} \phi_{X}, q_X) - (u_{aXY} \phi_{Y}, q) - (u_{aX} \phi_{Y}, q_Y)\\
\lesssim& \|u_{aXX}\|_\infty \|\phi_{X}\|_2 \|q\|_2 + \|u_{aX}\|_\infty \|\phi_{X}\|_2 \|q_X\|_2\\
&+ \|\widetilde{Y} u_{aXY}\|_\infty \|\phi_{Y}\|_2 \|\frac{q}{\widetilde{Y}}\|_2 + \|u_{aX}\|_\infty \|\phi_{Y}\|_2 \|q_Y\|_2\\ 
=& O(1)\|q_X\|_2^2 +  O(1)\|q_Y\|_2^2,
\end{aligned}
\end{equation}

\begin{equation}\label{1-7}
\begin{aligned}
(\Delta u_{aX} \phi, -q) =& -(u_{aXXX} \phi, q) - (u_{aXYY} \phi, q)\\
=& (u_{aXX} \phi_X, q) + (u_{aXX} \phi, q_X) + (u_{aXY} \phi_Y, q) + (u_{aXY} \phi, q_Y)\\
\lesssim& \|u_{aXX}\|_\infty \|\phi_{X}\|_2 \|q\|_2 + \|u_{aXX}\|_\infty \|\phi\|_2 \|q_X\|_2\\
&+ \|\widetilde{Y} u_{aXY}\|_\infty \|\phi_{Y}\|_2 \|\frac{q}{\widetilde{Y}}\|_2 + \|\widetilde{Y} u_{aXY}\|_\infty \|\frac{\phi}{\widetilde{Y}}\|_2 \|q_Y\|_2\\ 
=& O(1)\|q_X\|_2^2 +  O(1)\|q_Y\|_2^2.
\end{aligned}
\end{equation}
The last two terms of $J[\phi]$ contain some crucial terms. First we split
\begin{equation}\label{1-7-1}
(v_s \Delta \phi_Y , -q) = (v_a \Delta \phi_Y , -q) + (v_r \Delta \phi_Y , -q).
\end{equation}
For the first term of \eqref{1-7-1}, we can rearrange the terms as above:
\begin{equation}\label{1-8}
\begin{aligned}
(\ref{1-7-1}.1)=& -(v_a \phi_{XXY}, q) - (v_a \phi_{YYY}, q)\\
=& (v_{aX} \phi_{XY}, q) + (v_{a} \phi_{XY}, q_X) + (v_{aY} \phi_{YY}, q) + (v_{a} \phi_{YY}, q_Y)\\
=&-(v_{aXX} \phi_{Y}, q) - (v_{aX} \phi_{Y}, q_X) - (v_{aYY} \phi_{Y}, q) - (v_{aY} \phi_{Y}, q_Y)\\
&+ (v_{a}  (u_{aXY}q +  u_{aX}q_Y + u_{aY} q_X + u_a q_{XY}), q_X)\\
&+ (v_{a} (u_{aYY}q + 2 u_{aY}q_Y + u_a q_{YY}), q_Y).
\end{aligned}
\end{equation}
Similar as before, we can estimate
\begin{align*}
&|(\ref{1-8}.1)| + |(\ref{1-8}.2)| + |(\ref{1-8}.3)| + |(\ref{1-8}.4)|\\
\lesssim& \|v_{aXX}\|_\infty \|\phi_{Y}\|_2 \|q\|_2 + \|v_{aX}\|_\infty \|\phi_Y\|_2 \|q_X\|_2\\
&+  \|\widetilde{Y} v_{aYY}\|_\infty \|\phi_{Y}\|_2 \|\frac{q}{\widetilde{Y}}\|_2 +  \|v_{aY}\|_\infty \|\phi_Y\|_2 \|q_Y\|_2\\ 
=& O(1)\|q_X\|_2^2 +  O(1)\|q_Y\|_2^2.
\end{align*}
To estimate (\ref{1-8}.5) and (\ref{1-8}.6), we first note that by \eqref{assumption2} and \eqref{u_a_property}, 
$$
\|\frac{v_a}{\widetilde{Y}} \|_\infty \le  \|\frac{v_e^0}{\widetilde{Y}} \|_\infty + \|\frac{v_p}{\widetilde{Y}} \|_\infty \lesssim 1.
$$
Then
\begin{align*}
|(\ref{1-8}.5)| =& |(v_{a} u_{aXY}q , q_X) + (v_{a} u_{aX}q_Y , q_X)\\ 
&+ (v_{a} u_{aY} q_X, q_X) - \frac{1}{2} ((v_{a}u_a)_Y q_{X}), q_X)|\\
\lesssim& \|v_a\|_\infty \|\widetilde{Y} u_{aXY}\|_\infty  \|\frac{q}{\widetilde{Y}}\|_2\|q_{Y}\|_2 + \|v_{a}u_{aX}\|_\infty \|q_Y\|_2 \|q_X\|_2\\
&+ \|\frac{v_a}{\widetilde{Y}} \|_\infty \|\widetilde{Y} u_{aY}\|_\infty \|q_X\|_2^2 + (\|v_{aY}u_a\|_\infty +  \|\frac{v_a}{\widetilde{Y}} \|_\infty \|\widetilde{Y} u_{aY}\|_\infty )\|q_X\|_2^2\\
=& O(1)\|q_X\|_2^2 +  O(1)\|q_Y\|_2^2.
\end{align*}
Similarly,
\begin{align*}
|(\ref{1-8}.6)| =& |(v_{a}u_{aYY}q , q_Y) + 2 (v_{a} u_{aY}q_Y , q_Y) -\frac{1}{2} ((v_{a}u_a)_Y q_{Y}), q_Y)|\\
\lesssim& \|\frac{v_a}{\widetilde{Y}} \|_\infty \|\widetilde{Y}^2 u_{aYY}\|_\infty  \|\frac{q}{\widetilde{Y}}\|_2\|q_{Y}\|_2\\
&+ (\|v_{aY}u_a\|_\infty +  \|\frac{v_a}{\widetilde{Y}} \|_\infty \|\widetilde{Y} u_{aY}\|_\infty )\|q_Y\|_2^2\\
=& O(1)\|q_Y\|_2^2.
\end{align*}
For the second term of \eqref{1-7-1}, we have
\begin{equation}\label{1-8-1}
\begin{aligned}
(\ref{1-7-1}.2)=& -(v_r \phi_{XXY}, q) - (v_r \phi_{YYY}, q)\\
=& (v_{rX} \phi_{XY}, q) + (v_{r} \phi_{XY}, q_X) + (v_{rY} \phi_{YY}, q) + (v_{r} \phi_{YY}, q_Y)\\
=&-(v_{rXY} \phi_{X}, q) - (v_{rX} \phi_{X}, q_Y) - (v_{rYY} \phi_{Y}, q) - (v_{rY} \phi_{Y}, q_Y)\\
&+ (v_{r}  (u_{aXY}q +  u_{aX}q_Y + u_{aY} q_X + u_a q_{XY}), q_X)\\
&+ (v_{r} (u_{aYY}q + 2 u_{aY}q_Y + u_a q_{YY}), q_Y).
\end{aligned}
\end{equation}
The crucial term here is $(\ref{1-8-1}.2)$, as the rest can be controlled by $\|q_X\|_2^2 + \|q_Y\|_2^2$ by the similar argument above. For this crucial term, we will use \eqref{u_r_property} and \eqref{Hardy_2} with $G = q_X$ and $q_Y$ to estimate
\begin{align*}
|(\ref{1-8-1}.2)| =& |(v_{rX} u_{aX}q , q_Y) + (v_{rX} u_a q_X, q_Y)|\\
\lesssim& \|Xv_{rX}\|_{\infty} \|u_{aX}\|_\infty \| \frac{q}{X} \|_2 \|q_Y\|_2  + \|X^2v_{rX}\|_\infty \|\sqrt{u_a} \frac{q_X}{X}\|_2 \|\sqrt{u_a} \frac{q_Y}{X}\|_2\\
=&O(1)\|q_X\|_2\|q_Y\|_2 + O(\|h\|_\infty) \varepsilon \|\sqrt{u_a} q_{XX}\|_2\|\sqrt{u_a} q_{XY}\|_2.
\end{align*}
For other terms, we can estimate similarly as before:
\begin{align*}
&|(\ref{1-8-1}.1)|+|(\ref{1-8-1}.3)|+|(\ref{1-8-1}.4)|\\
\lesssim&  \|Xv_{rXY}\|_{\infty} \| \phi_X \|_2 \| \frac{q}{X} \|_2 +  \|v_{rYY}\|_{\infty} \| \phi_Y \|_2 \| q \|_2 + \|v_{rY}\|_{\infty} \| \phi_Y \|_2 \| q_Y \|_2\\
=& O(1)\|q_X\|_2^2 +  O(1)\|q_Y\|_2^2.
\end{align*}
For the last two terms of \eqref{1-8-1}, we use the fact that $v_r$ is compactly supported in $Y \in (-1,1)$, so its interaction with any contribution from the Prandtl corrector is small when $\varepsilon$ is small enough. We have
$$
(\ref{1-8-1}.5) + (\ref{1-8-1}.6) = O(1)\|q_X\|_2^2 +  O(1)\|q_Y\|_2^2.
$$
Therefore,
\begin{equation}\label{1-9}
(\ref{1-7-1}.2) = O(1)\|q_X\|_2^2 +  O(1)\|q_Y\|_2^2 + O(\|h\|_\infty) \varepsilon \|\sqrt{u_a} q_{XX}\|_2^2 + O(\|h\|_\infty) \varepsilon \|\sqrt{u_a} q_{XY}\|_2^2.
\end{equation}
Finally, we estimate
\begin{align}\label{1-9-1}
(- \Delta v_s \phi_Y, -q) =& (\Delta v_a \phi_Y, q) + (\Delta v_r \phi_Y, q).
\end{align}
Estimate for $(\ref{1-9-1}.1)$ is straightforward:
\begin{align*}
|(\ref{1-9-1}.1)| =& |( v_{aXX} \phi_Y, q) + ( v_{aYY} \phi_Y, q)|\\
\lesssim& \|v_{aXX}\|_{\infty} \| \phi_Y \|_2 \| q \|_2 + \|\widetilde{Y} v_{aYY}\|_\infty \|\phi_{Y}\|_2 \|\frac{q}{\widetilde{Y}}\|_2\\
=&O(1)\|q_Y\|_2^2.
\end{align*}
For $(\ref{1-9-1}.2)$, we have
\begin{align*}
|(\ref{1-9-1}.2)| =& |( v_{rXX} \phi_Y, q) + ( v_{rYY} \phi_Y, q)| \\
=&  |( v_{rX} \phi_{XY}, q) + ( v_{rX} \phi_Y, q_X)| +  O(1)\|q_Y\|_2^2\\
\le & |( v_{rXY} \phi_{X}, q)| + |( v_{rX} \phi_{X}, q_Y )| +|( v_{rX} \phi_Y, q_X)| +  O(1)\|q_Y\|_2^2\\
\lesssim & \|Xv_{rXY}\|_{\infty} \|\phi_{X}\|_2 \|\frac{q}{X}\|_2 + |( v_{rX} \phi_{X}, q_Y )| +|( v_{rX} \phi_Y, q_X)|+  O(1)\|q_Y\|_2^2.
\end{align*}
$( v_{rX} \phi_{X}, q_Y )$ is exactly $(\ref{1-8-1}.2)$ which has been estimated above, and the term $( v_{rX} \phi_Y, q_X)$ can be handled similarly. Therefore,
\begin{equation}\label{1-10}
\begin{aligned}
(- \Delta v_s \phi_Y, -q) =&O(1)\|q_X\|_2^2 +  O(1)\|q_Y\|_2^2\\
& + O(\|h\|_\infty) \varepsilon \|\sqrt{u_a} q_{XX}\|_2^2 + O(\|h\|_\infty) \varepsilon \|\sqrt{u_a} q_{XY}\|_2^2.
\end{aligned}
\end{equation}
Now collecting all terms above, and using \eqref{Hardy_1} with $F = q_X$ and $q_Y$, we have
\begin{align*}
&\varepsilon (u_a q_{XX}, q_{XX}) + 2\varepsilon (u_a q_{XY}, q_{XY}) + \varepsilon (u_a q_{YY}, q_{YY}) \\
&-\varepsilon (u_{aY}q_{Y}, q_{Y})|_{Y=1} + \varepsilon (u_{aY}q_{Y}, q_{Y})|_{Y=-1}\\
\lesssim &\|q_X\|_2^2 + \|q_Y\|_2^2 + \varepsilon \|h\|_\infty\|\sqrt{u_a} q_{XX}\|_2^2  + \varepsilon \|h\|_\infty\|\sqrt{u_a} q_{XY}\|_2^2 + \|f\|_2^2\\
\lesssim & \sigma \varepsilon \|\sqrt{u_a} q_{XY}\|_2^2 + \frac{1}{\sigma^2}\|u_a q_X\|_2^2 + \sigma \varepsilon \|\sqrt{u_a} q_{YY}\|_2^2 + \frac{1}{\sigma^2}\|u_a q_Y\|_2^2\\
&+ \varepsilon \|h\|_\infty\|\sqrt{u_a} q_{XX}\|_2^2  + \varepsilon \|h\|_\infty\|\sqrt{u_a} q_{XY}\|_2^2 + \|f\|_2^2.
\end{align*}
Since $\|h\|_\infty \ll 1$, by choosing $\sigma$ small enough, we have proved \eqref{energy_1}.
\end{proof}

Next we will close the estimate \eqref{energy_1} by testing the equation \eqref{linear_quotient} by $qX$. We would like to re-emphasis that the weight $X$ interacts favorably with the profile $(u_r,v_r)$, as it cancels one derivative of $(u_r,v_r)$ in $X$. We first consider the case when $L \ll 1$.

\begin{lemma}\label{lem2}
Let $q$ be a solution of \eqref{linear_quotient}. Assume that $\varepsilon \ll L \ll1$ and $\|h\|_\infty \ll 1$, we have
\begin{equation}\label{energy_2}
\begin{aligned}
&\|u_a q_X\|_2^2 + \|u_a q_Y\|_2^2\\ 
\lesssim& (L + \sqrt{\varepsilon} + \|h\|_\infty)^{\frac{4}{3}} \varepsilon \left( \| \sqrt{u_a} q_{XX}\|_2^2 +\| \sqrt{u_a} q_{XY}\|_2^2 + \| \sqrt{u_a} q_{YY}\|_2^2 \right) + \|f\|_2^2.
\end{aligned}
\end{equation}
\end{lemma}

\begin{proof}
We take inner product of the equation \eqref{linear_quotient} with $qX$, and estimate term by term. 

\textbf{Step 1:} Rayleigh terms.

\begin{equation}\label{2-1}
\begin{aligned}
((u_a^2 q_{X})_{XX} , qX) &= -((u_a^2 q_{X})_{X} , q) - ((u_a^2 q_{X})_{X} , q_{X}X)\\
&= 2(u_a^2 q_X, q_X) + (u_a^2 q_{X}, q_{XX}X)\\
&= \frac{3}{2}(u_a^2 q_X, q_X) - (u_a u_{aX} q_X, q_XX)\\
&= \frac{3}{2}(u_a^2 q_X, q_X) + O(L)\| q_X\|_2^2,
\end{aligned}
\end{equation}
where we used $|X|\le L$.
Similarly, we have
\begin{equation}\label{2-2}
\begin{aligned}
((u_a^2 q_{Y})_{XY} , qX) &= -((u_a^2 q_{Y})_{X} , q_YX)\\ 
&= - (u_a^2 q_{XY} , q_YX)- 2(u_a u_{aX} q_Y, q_YX)\\
&=  \frac{1}{2}(u_a^2 q_Y, q_Y) -(u_a u_{aX} q_Y, q_Y X)\\
&=  \frac{1}{2}(u_a^2 q_Y, q_Y) + O(L)\|q_Y\|_2^2.
\end{aligned}
\end{equation}

\textbf{Step 2:} $\Delta^2 \phi$ term.

For the Bi-Laplacian term, we split it into three parts again. For the first part,
\begin{equation}\label{2-2-0.5}
\begin{aligned}
-\varepsilon (\phi_{XXXX}, qX) =& \varepsilon (\phi_{XXX}, q) + \varepsilon (\phi_{XXX}, q_XX) \\
=&-2\varepsilon (\phi_{XX}, q_X) - \varepsilon (\phi_{XX}, q_{XX}X)\\
=&-2\varepsilon (u_{aXX}q + 2 u_{aX}q_X + u_a q_{XX}, q_{X})\\
&- \varepsilon (u_{aXX}q + 2 u_{aX}q_X + u_a q_{XX}, q_{XX}X).
\end{aligned}
\end{equation}
Since $X < 0$, the last term $- \varepsilon( u_a q_{XX}, q_{XX}X)$ is a nonnegative term. The rest can be easily controlled by $O(\varepsilon)\|q_X\|_2^2$ by the Poincare inequality $\|q\|_2 \lesssim \|q_X\|_2$.

For the second part,
\begin{equation}\label{2-2-1}
\begin{aligned}
-2\varepsilon (\phi_{XXYY}, qX) =& -2\varepsilon (\phi_{XY}, q_{XY}X)-2\varepsilon (\phi_{XY}, q_{Y})\\
=&-2\varepsilon (u_{aXY}q +  u_{aX}q_Y + u_{aY} q_X + u_a q_{XY}, q_{XY}X)\\
&-2\varepsilon (u_{aXY}q +  u_{aX}q_Y + u_{aY} q_X + u_a q_{XY}, q_{Y})\\
=&-2\varepsilon (u_a q_{XY}, q_{XY}X) +2\varepsilon (u_{aXYY} q, q_{X}X) + 2\varepsilon (u_{aXY} q_{Y}, q_{X}X)\\
&+ \varepsilon(u_{aXX}q_Y,q_Y X)+ 2\varepsilon (u_{aX} q_{Y}, q_{Y})+ \varepsilon(u_{aYY}q_X,q_XX) \\
&-2\varepsilon (u_{aXY}q +  u_{aX}q_Y + u_{aY} q_X, q_{Y}).
\end{aligned}
\end{equation}
As above, the first term $-2\varepsilon (u_a q_{XY}, q_{XY}X)$ is a favorable term.  Using \eqref{u_a_property} and $|X|\le L$, we can immediately estimate
\begin{align*}
& |(\ref{2-2-1}.3)| + |(\ref{2-2-1}.4)|+ |(\ref{2-2-1}.5)|+|(\ref{2-2-1}.6)|\\
\lesssim & \sqrt{\varepsilon}  \|\sqrt{\varepsilon} u_{aXY}\|_{\infty} \|\phi_{Y}\|_2 \|\phi_{X}\|_2 + \varepsilon \| u_{aXX}\|_{\infty} \|\phi_{Y}\|_2^2\\
 &+ \varepsilon \| u_{aX}\|_{\infty} \|\phi_{Y}\|_2^2
+ L \| \varepsilon u_{aYY}\|_{\infty} \|\phi_{X}\|_2^2\\
=& O(\sqrt{\varepsilon} + L)\|q_X\|_2^2 + O(\sqrt\varepsilon)\|q_Y\|_2^2.
\end{align*}
For the last two terms of \eqref{2-2-1}, we use the Poincare inequality and Hardy's inequality to obtain
\begin{align*}
& |(\ref{2-2-1}.2)| + |(\ref{2-2-1}.7)|\\
\lesssim & \sqrt{\varepsilon}  \|\sqrt\varepsilon \widetilde{Y} u_{aXYY}\|_{\infty} \|\frac{q}{\widetilde{Y}}\|_2\|q_X\|_2 + \sqrt{\varepsilon}  \|\sqrt{\varepsilon} u_{aXY}\|_{\infty} \|q\|_2\|q_Y\|_2\\
&+ \varepsilon \| u_{aX}\|_{\infty} \|q_Y\|_2^2 + \sqrt{\varepsilon}  \|\sqrt{\varepsilon} u_{aY}\|_{\infty} \|q_X\|_2\|q_Y\|_2\\
=& O(\sqrt{\varepsilon})\|q_X\|_2^2 + O(\sqrt\varepsilon)\|q_Y\|_2^2.
\end{align*}
For the last part, note that we do not have $q_Y = 0$ when $Y = \pm 1$,
\begin{equation}\label{2-2-2}
\begin{aligned}
-\varepsilon (\phi_{YYYY}, qX) =& \varepsilon (\phi_{YYY}, q_{Y}X)\\
=& \varepsilon (\phi_{YY}, q_{Y}X)|_{Y=1} - \varepsilon (\phi_{YY}, q_{Y}X)|_{Y=-1} - \varepsilon(\phi_{YY},q_{YY}X)\\
=& 2\varepsilon (u_{aY}q_{Y}, q_{Y}X)|_{Y=1} - 2\varepsilon (u_{aY}q_{Y}, q_{Y}X)|_{Y=-1}\\
&-\varepsilon (u_{aYY}q + 2 u_{aY}q_Y + u_a q_{YY}, q_{YY}X)\\
=&-\varepsilon (u_a q_{YY}, q_{YY}X) + 2\varepsilon(u_{aYY}q_Y,q_YX) + \varepsilon (u_{aYYY}q, q_YX)\\
&+\varepsilon (u_{aY}q_{Y}, q_{Y}X)|_{Y=1} - \varepsilon (u_{aY}q_{Y}, q_{Y}X)|_{Y=-1}.
\end{aligned}
\end{equation}
As above, the first term $-\varepsilon (u_a q_{YY}, q_{YY}X)$ is a nonnegative term. 
Since $u_{aY} > 0$ at $Y=-1$ and $u_{aY} < 0$ at $Y=1$, the last two terms $(\ref{2-2-2}.4)$ and $(\ref{2-2-2}.5)$ are nonnegative as well. $(\ref{2-2-2}.2)$ can be immediately estimated using \eqref{u_a_property} and $|X| \le L$:
$$
|(\ref{2-2-2}.2)| \lesssim L \|\varepsilon u_{aYY}\|_{\infty} \|q_Y\|_2^2 = O(L)\|q_Y\|_2^2.
$$
Finally, we use \eqref{u_a_property}, Hardy inequality, and $|X| \le L$ to estimate
$$
|(\ref{2-2-2}.3)| \lesssim L  \|\varepsilon \widetilde{Y} u_{aYYY}\|_{\infty} \|\frac{q}{\widetilde{Y}}\|_2 \|q_Y\|_2= O(L)\|q_Y\|_2^2.
$$
Therefore,
\begin{equation}\label{2-3}
\begin{aligned}
(-\varepsilon\Delta^2 \phi,  qX) =& -\varepsilon (u_a q_{XX}, q_{XX}X) -2\varepsilon (u_a q_{XY}, q_{XY}X) - \varepsilon (u_a q_{YY}, q_{YY}X)\\
& +\varepsilon (u_{aY}q_{Y}, q_{Y}X)|_{Y=1} - \varepsilon (u_{aY}q_{Y}, q_{Y}X)|_{Y=-1}\\
&+O(\sqrt{\varepsilon} + L)\|q_X\|_2^2 + O(\sqrt{\varepsilon} + L)\|q_Y\|_2^2.
\end{aligned}
\end{equation}

\textbf{Step 3:} $J[\phi]$ terms.

We estimate $(J[\phi],qX)$ term by term. First, we have
\begin{equation}\label{2-4}
\begin{aligned}
( u_r \Delta \phi_X, qX) =& ( u_r \phi_{XXX}, qX) + ( u_r \phi_{XYY}, qX)\\
=& -(u_{rX} \phi_{XX}, qX) -(u_{r} \phi_{XX}, q_XX) - (u_{r} \phi_{XX}, q)\\
& - (u_{rY} \phi_{XY}, qX)  - (u_{r} \phi_{XY}, q_YX)\\
=& (u_{rXX} \phi_{X}, qX) + (u_{rX} \phi_{X}, q_XX) + (u_{rX} \phi_{X}, q)\\
& (u_{rX} \phi_X,q) + (u_r \phi_X, q_X)  + (u_{rYY} \phi_{X}, qX) - (u_{rY} \phi_{X}, q_YX)\\
&- (u_{r} (u_{aXX}q + 2 u_{aX}q_X + u_a q_{XX}), q_XX)\\
&- (u_{r} (u_{aXY}q +  u_{aX}q_Y + u_{aY} q_X + u_a q_{XY}), q_YX).
\end{aligned}
\end{equation}
By \eqref{u_r_property} and Hardy's inequality in $X$, we can immediately estimate
$$
\sum_{n=1}^7 |(\ref{2-4}.n)| = O(\varepsilon) \|q_X\|_2^2 + O(\varepsilon)\|q_Y\|_2^2.
$$
Next, we have
\begin{align*}
|(\ref{2-4}.8)| =& |(u_{r}u_{aXX}q , q_XX) + 2 (u_{r}u_{aX}q_X , q_XX)\\
 &-\frac{1}{2} ((u_{r}u_a)_X q_{X}, q_XX) - \frac{1}{2} (u_{r}u_a q_{X}, q_X) |\\
=& O(\varepsilon) \|q_X\|_2^2.
\end{align*}
$(\ref{2-4}.9)$ can be controlled by $O(\varepsilon) \|q_X\|_2^2 + O(\varepsilon)\|q_Y\|_2^2$ in a similar manner.

Similarly, we can estimate 
\begin{equation}\label{2-5}
\begin{aligned}
(- \Delta u_r \phi_X, qX) =& -(u_{rXX} \phi_X, qX) - (u_{rYY} \phi_X, qX)\\
\lesssim & \|X^2u_{rXX}\|_\infty \|\phi_X\|_2\|\frac{q}{X}\|_2 +\|Xu_{rYY}\|_\infty \|\phi_X\|_2\|q\|_2\\
=& O(\varepsilon)\|q_X\|_2^2,
\end{aligned}
\end{equation}

\begin{equation}\label{2-6}
\begin{aligned}
(- u_{aX} \Delta \phi , qX) =& -(u_{aX} \phi_{XX}, qX) - (u_{aX} \phi_{YY}, qX)\\
=& (u_{aXX} \phi_{X}, qX) + (u_{aX} \phi_{X}, q_XX)+ (u_{aX} \phi_{X}, q)\\
& + (u_{aXY} \phi_{Y}, qX) + (u_{aX} \phi_{Y}, q_YX)\\
=& O(L)\|q_X\|_2^2 +  O(L)\|q_Y\|_2^2,
\end{aligned}
\end{equation}

\begin{equation}\label{2-7}
\begin{aligned}
(\Delta u_{aX} \phi, qX) =& (u_{aXXX} \phi, qX) + (u_{aXYY} \phi, qX)\\
=& -(u_{aXX} \phi_X, qX) - (u_{aXX} \phi, q_XX) - (u_{aXX} \phi, q)\\
& - (u_{aXY} \phi_Y, qX) - (u_{aXY} \phi, q_YX)\\
=& O(L)\|q_X\|_2^2 +  O(L)\|q_Y\|_2^2.
\end{aligned}
\end{equation}

Next, we split
\begin{equation}\label{2-7-1}
(v_s \Delta \phi_Y , qX) = (v_a \Delta \phi_Y , qX) + (v_r \Delta \phi_Y , qX).
\end{equation}
For the first term $(\ref{2-7-1}.1)$, we have
\begin{equation}\label{2-8}
\begin{aligned}
(v_a \Delta \phi_Y , qX)=& (v_a \phi_{XXY}, qX) + (v_a \phi_{YYY}, qX)\\
=& -(v_{aX} \phi_{XY}, qX) - (v_{a} \phi_{XY}, q_XX)  - (v_{a} \phi_{XY}, q)\\
&- (v_{aY} \phi_{YY}, qX) - (v_{a} \phi_{YY}, q_YX)\\
=&(v_{aXY} \phi_{X}, qX) + (v_{aX} \phi_{X}, q_YX)  + (v_{aX} \phi_{Y}, q)\\
&+(v_a \phi_Y, q_X) + (v_{aYY} \phi_{Y}, qX) + (v_{aY} \phi_{Y}, q_YX)\\
&- (v_{a}  (u_{aXY}q +  u_{aX}q_Y + u_{aY} q_X + u_a q_{XY}), q_XX)\\
&- (v_{a} (u_{aYY}q + 2 u_{aY}q_Y + u_a q_{YY}), q_YX).
\end{aligned}
\end{equation}
Immediately we can estimate
\begin{align*}
&|(\ref{2-8}.2)| + |(\ref{2-8}.4)| + |(\ref{2-8}.6)|\\
\lesssim &  L\|v_{aX}\|_\infty \|\phi_X\|_2 \|q_Y\|_2 + \|v_{a}\|_\infty \|\phi_Y\|_2 \|q_X\|_2 + L\|v_{aY}\|_\infty \|\phi_Y\|_2 \|q_Y\|_2\\
=& O(\|v_e^0\|_\infty + \sqrt{\varepsilon} + L) (\|q_X\|_2^2 + \|q_Y\|_2^2).
\end{align*}
We use Poincare or Hardy's inequality to estimate
\begin{align*}
&|(\ref{2-8}.1)| + |(\ref{2-8}.3)| + |(\ref{2-8}.5)|\\
\lesssim &  L\|v_{aXY}\|_\infty \|\phi_X\|_2 \|q\|_2 + L\|v_{aX}\|_\infty \|\phi_Y\|_2 \|\frac{q}{X}\|_2 + L\|\widetilde{Y} v_{aYY}\|_\infty \|\phi_Y\|_2 \|\frac{q}{\widetilde{Y}}\|_2 \\
=& O(L) (\|q_X\|_2^2 + \|q_Y\|_2^2).
\end{align*}
For $(\ref{2-8}.7)$, we use integration by parts in $Y$, \eqref{assumption2}, and \eqref{u_a_property} to estimate
\begin{align*}
|(\ref{2-8}.7)| \lesssim & L \|v_a\|_\infty \|\widetilde{Y} u_{aXY}\|_\infty \|q_X\|_2 \|\frac{q}{\widetilde{Y}}\|_2 + L\|v_a\|_\infty \| u_{aX}\|_\infty \|q_Y\|_2\|q_X\|_2\\
&+ L \|\frac{v_a}{\widetilde{Y}}\|_\infty \|\widetilde{Y} u_{aY}\|_\infty \|q_X\|_2^2 + L \|v_{aY}\|_\infty \| u_{a}\|_\infty \|q_X\|_2^2\\
=& O(L) (\|q_X\|_2^2 + \|q_Y\|_2^2).
\end{align*}
Similarly, we can estimate the last term
$$
|(\ref{2-8}.8)| = O(L) \|q_Y\|_2^2.
$$

For the second term of \eqref{2-7-1}, we have
\begin{align*}
(v_r \Delta \phi_Y , qX)=& (v_r \phi_{XXY}, qX) + (v_r \phi_{YYY}, qX)\\
=& -(v_{rX} \phi_{XY}, qX) - (v_{r} \phi_{XY}, q_XX)  - (v_{r} \phi_{XY}, q)\\
&- (v_{rY} \phi_{YY}, qX) - (v_{r} \phi_{YY}, q_YX)\\
=&(v_{rXY} \phi_{X}, qX) + (v_{rX} \phi_{X}, q_YX)  + (v_{rX} \phi_{Y}, q)\\
&+(v_r \phi_Y, q_X) + (v_{rYY} \phi_{Y}, qX) + (v_{rY} \phi_{Y}, q_YX)\\
&- (v_{r}  (u_{aXY}q +  u_{aX}q_Y + u_{aY} q_X + u_a q_{XY}), q_XX)\\
&- (v_{r} (u_{aYY}q + 2 u_{aY}q_Y + u_a q_{YY}), q_YX).\\
\end{align*}
Note that unlike the crucial term in Lemma \ref{lem_1}, $(v_{rX} \phi_{X}, q_YX)$ can be straightforwardly estimated as 
$$(v_{rX} \phi_{X}, q_YX) \lesssim \|X v_{rX}\|_\infty \|q_X\|_2 \|q_Y\|_2 =  O(\|h\|_\infty) \|q_X\|_2 \|q_Y\|_2$$
due to the presence of weight $X$. We can estimate the rest by similar argument as the terms in \eqref{2-8} and the fact that the interaction between $v_r$ and the Prandtl layers are small, to obtain
\begin{equation}\label{2-9}
(v_r \Delta \phi_Y , qX) = O(\varepsilon + \|h\|_{\infty}) (\|q_X\|_2^2 + \|q_Y\|_2^2).
\end{equation}

Finally, we estimate the last term of $J[\phi]$:
\begin{equation}\label{2-10}
\begin{aligned}
(- \Delta v_s \phi_Y, qX) =& -(v_{aXX} \phi_Y, qX) -(v_{aYY} \phi_Y, qX)\\ 
&-(v_{rXX} \phi_Y, qX) -(v_{rYY} \phi_Y, qX)\\
=& O(L)\|q_X\|_2\|q_Y\|_2 +  O(L)\|q_Y\|_2^2 \\
&+ O( \|h\|_{\infty})\|q_X\|_2\|q_Y\|_2 + O(\varepsilon)\|q_X\|_2\|q_Y\|_2.
\end{aligned}
\end{equation}
Now collecting \eqref{2-1}-\eqref{2-10}, and using \eqref{Hardy_1} with $F = q_X$ and $q_Y$, we have
\begin{align*}
&\frac{3}{2}\|u_a q_X\|_2^2 + \frac{1}{2}\|u_a q_Y\|_2^2 - \varepsilon (u_a q_{XX}, q_{XX}X) - 2\varepsilon (u_a q_{XY}, q_{XY}X)\\
& - \varepsilon (u_a q_{YY}, q_{YY}X)
+\varepsilon (u_{aY}q_{Y}, q_{Y}X)|_{Y=1} - \varepsilon (u_{aY}q_{Y}, q_{Y}X)|_{Y=-1}\\
=& O(L + \sqrt{\varepsilon} + \|h\|_\infty) (\|q_X\|_2^2 + \|q_Y\|_2^2) + \|f\|_2^2\\
=& O(L + \sqrt{\varepsilon} + \|h\|_\infty) \Big( \sigma \varepsilon \|\sqrt{u_a} q_{XY}\|_2^2 + \frac{1}{\sigma^2}\|u_a q_X\|_2^2 + \sigma \varepsilon \|\sqrt{u_a} q_{YY}\|_2^2 \\ 
&+ \frac{1}{\sigma^2}\|u_a q_Y\|_2^2 \Big)+ \|f\|_2^2.
\end{align*}
Choosing $\sigma = (L + \sqrt{\varepsilon} + \|h\|_\infty)^{\frac{1}{3}}$, we will obtain the desired estimate \eqref{energy_2}.
\end{proof}

Our next lemma addresses the case for arbitrary $L>0$, under the additional assumptions that the background Euler flow is a shear flow, and the first Prandtl layer $u^0_p$ is concave. This lemma is based on the delicate cancellations observed in \cites{Iyer20,IyerMasmoudi21a,GaoZhang23} when $u^0_p$ is concave.

\begin{lemma}\label{lem3}
Let $L > 0$, $q$ be a solution of \eqref{linear_quotient}. Assume that $\varepsilon \ll 1$, $\|h\|_\infty \ll 1$, $v_e^0 = 0$ and $u_p^0$ satisfies \eqref{concave}, we have
\begin{equation}\label{energy_3}
\begin{aligned}
&\|u_a q_X\|_2^2 + \|u_a q_Y\|_2^2\\ 
\lesssim& (\sqrt{\varepsilon} + \|h\|_\infty)^{\frac{4}{3}} \varepsilon \left( \| \sqrt{u_a} q_{XX}\|_2^2 +\| \sqrt{u_a} q_{XY}\|_2^2 + \| \sqrt{u_a} q_{YY}\|_2^2 \right) + \|f\|_2^2.
\end{aligned}
\end{equation}
\end{lemma}
\begin{proof}
We take inner product of the equation \eqref{linear_quotient} with $qX$, and estimate term by term.  Some of the terms are estimated exactly as in Lemma \ref{lem2}. We will mainly focus on the differences.

\textbf{Step 1:} Rayleigh terms.

As in \eqref{2-1} and \eqref{2-2}, we have
\begin{equation}\label{3-1}
((u_a^2 q_{X})_{XX} , qX) = \frac{3}{2}(u_a^2 q_X, q_X) - (u_a u_{aX} q_X, q_XX),
\end{equation}
and
\begin{equation}\label{3-2}
((u_a^2 q_{Y})_{XY} , qX) =  \frac{1}{2}(u_a^2 q_Y, q_Y) -(u_a u_{aX} q_Y, q_Y X).
\end{equation}

\textbf{Step 2:} $\Delta^2 \phi$ term.

For the Bi-Laplacian term, we split it into three parts again. For the first part, we estimate exactly as in \eqref{2-2-0.5}:
\begin{equation}\label{3-3}
-\varepsilon (\phi_{XXXX}, qX) = - \varepsilon( u_a q_{XX}, q_{XX}X) + O(\varepsilon)\|q_X\|_2^2.
\end{equation}
For the second part, we have
\begin{equation}\label{3-4}
\begin{aligned}
-2\varepsilon (\phi_{XXYY}, qX) =& -2\varepsilon (u_a q_{XY}, q_{XY}X) +2\varepsilon (u_{aXYY} q, q_{X}X) + 2\varepsilon (u_{aXY} q_{Y}, q_{X}X)\\
&+ \varepsilon(u_{aXX}q_Y,q_Y X)+ 2\varepsilon (u_{aX} q_{Y}, q_{Y})+ \varepsilon(u_{aYY}q_X,q_XX) \\
&-2\varepsilon (u_{aXY}q +  u_{aX}q_Y + u_{aY} q_X, q_{Y}).
\end{aligned}
\end{equation}
We keep the terms $(\ref{3-4}.1)$ and $(\ref{3-4}.6)$, and estimate the rest exactly as in \eqref{2-2-1}. We have
\begin{equation}\label{3-4-1}
\begin{aligned}
-2\varepsilon (\phi_{XXYY}, qX) =& -2\varepsilon (u_a q_{XY}, q_{XY}X) + \varepsilon(u_{aYY}q_X,q_XX) \\
&+ O(\sqrt{\varepsilon})(\|q_X\|_2^2 + \|q_Y\|_2^2).
\end{aligned}
\end{equation}
For the last part, as in \eqref{2-2-2}, we have
\begin{equation}\label{3-5}
\begin{aligned}
-\varepsilon (\phi_{YYYY}, qX) =& -\varepsilon (u_a q_{YY}, q_{YY}X) + 2\varepsilon(u_{aYY}q_Y,q_YX) - \frac{1}{2} \varepsilon (u_{aYYYY}q, qX)\\
&+\varepsilon (u_{aY}q_{Y}, q_{Y}X)|_{Y=1} - \varepsilon (u_{aY}q_{Y}, q_{Y}X)|_{Y=-1}.
\end{aligned}
\end{equation}

\textbf{Step 3:} $J[\phi]$ terms.

Now we move to the $J[\phi]$ terms. For the first two terms, we estimate exactly as in \eqref{2-4} and \eqref{2-5} to obtain
\begin{equation}\label{3-6}
|(u_r \Delta \phi_X, qX)| + |(- \Delta u_r \phi_X, qX)| = O(\varepsilon)\|q_X\|_2^2.
\end{equation}

Next we move to $(-u_{aX} \Delta \phi , qX)$. From \eqref{2-6}, we have
\begin{equation}\label{3-7}
\begin{aligned}
(- u_{aX} \Delta \phi , qX)
=& (u_{aXX} \phi_{X}, qX) + (u_{aX} \phi_{X}, q_XX)+ (u_{aX} \phi_{X}, q)\\
& + (u_{aXY} \phi_{Y}, qX) + (u_{aX} \phi_{Y}, q_YX).
\end{aligned}
\end{equation}
We can estimate
\begin{align*}
|(\ref{3-7}.1)| + |(\ref{3-7}.3)| \lesssim& (\| \widetilde{Y} u_{aXX}\|_\infty + \| \widetilde{Y} u_{aX}\|_\infty )  \|\phi_X\|_2 \|\frac{q}{\widetilde{Y}}\|_2\\
=& O(\sqrt{\varepsilon}) \|q_X\|_2 \|q_Y\|_2,
\end{align*}
where we used the fact that the background Euler is a shear flow so that $u_{aX} = u_{pX}^0 + O(\sqrt{\varepsilon})$, \eqref{u_a_property}, and the Hardy's inequality. For the rest, we write
\begin{align*}
&(\ref{3-7}.2) + (\ref{3-7}.4) + (\ref{3-7}.5)\\
=& (u_{aX} u_a q_X, q_XX) + (u_{aX}^2 q, q_XX) + (u_{aXY} u_a q_Y, qX)\\
&+ (u_{aXY} u_{aY} q, qX) + (u_{aX} u_{aY} q, q_YX) + (u_{aX} u_a q_Y, q_YX)\\
=& (u_{aX} u_a q_X, q_XX) + (u_{aX} u_a q_Y, q_YX) - \frac{1}{2} ((u_{aXY} u_a)_Y q, qX)\\
&+ (u_{aXY} u_{aY} q, qX) - \frac{1}{2 }  ((u_{aX} u_{aY})_Y q, qX) + O(1)\| \widetilde{Y} u_{aX}^2\|_\infty \|q_X\|_2 \|\frac{q}{\widetilde{Y}}\|_2\\
=& (u_{aX} u_a q_X, q_XX) + (u_{aX} u_a q_Y, q_YX)\\
& - \frac{1}{2 } ((u_{aXYY} u_a + u_{aX}u_{aYY}) q,qX) + O(\sqrt{\varepsilon}) \|q_X\|_2 \|q_Y\|_2.
\end{align*}
Therefore,
\begin{equation}\label{3-7-1}
\begin{aligned}
(- u_{aX} \Delta \phi , qX) =& (u_{aX} u_a q_X, q_XX) + (u_{aX} u_a q_Y, q_YX)\\
& - \frac{1}{2 } ((u_{aXYY} u_a + u_{aX}u_{aYY}) q,qX) + O(\sqrt{\varepsilon}) \|q_X\|_2 \|q_Y\|_2.
\end{aligned}
\end{equation}

For the next term in $J[\phi]$, we have
\begin{equation}\label{3-8}
\begin{aligned}
(\Delta u_{aX} \phi, qX) =& (u_{aXXX} \phi, qX) + (u_{aXYY} \phi, qX)\\
=& (u_{aXYY} u_a q, qX) + O(1)\|\widetilde{Y}^2 u_{aXXX}\|_\infty \|\frac{\phi}{\widetilde{Y}}\|_2 \|\frac{q}{\widetilde{Y}}\|_2 \\
=& (u_{aXYY} u_a q, qX) + O(\varepsilon)\|q_Y\|_2^2.
\end{aligned}
\end{equation}

Next, we split
\begin{equation*}
(v_s \Delta \phi_Y , qX) = (v_a \Delta \phi_Y , qX) + (v_r \Delta \phi_Y , qX).
\end{equation*}
For the later term, we use the estimate \eqref{2-9}:
\begin{equation}\label{3-9}
(v_r \Delta \phi_Y , qX) = O(\varepsilon + \|h\|_{\infty}) (\|q_X\|_2^2 + \|q_Y\|_2^2).
\end{equation}
And for the first term, as in \eqref{2-8}, we have
\begin{equation}\label{3-10}
\begin{aligned}
(v_a \Delta \phi_Y , qX)=&(v_{aXY} \phi_{X}, qX) + (v_{aX} \phi_{X}, q_YX)  + (v_{aX} \phi_{Y}, q)\\
&+(v_a \phi_Y, q_X) + (v_{aYY} \phi_{Y}, qX) + (v_{aY} \phi_{Y}, q_YX)\\
&- (v_{a}  (u_{aXY}q +  u_{aX}q_Y + u_{aY} q_X + u_a q_{XY}), q_XX)\\
&- (v_{a} (u_{aYY}q + 2 u_{aY}q_Y + u_a q_{YY}), q_YX).
\end{aligned}
\end{equation}
The first four terms of \eqref{3-10} can be controlled by $O(\sqrt{\varepsilon})\|q_X\|_2\|q_Y\|_2$ through \eqref{u_a_property} and Hardy's inequality.
\begin{align*}
(\ref{3-10}.5) + (\ref{3-10}.6) =& (v_{aYY} u_{aY} q, qX) + (v_{aYY} u_a q_Y, qX)\\
&+ (v_{aY} u_{aY} q, q_YX) + (v_{aY} u_a q_Y, q_YX)\\
=& -\frac{1}{2} (v_{aYYY} u_a q, qX) - \frac{1}{2 }(v_{aY} u_{aYY} q, qX) + (v_{aY} u_a q_Y, q_YX),
\end{align*}
\begin{align*}
(\ref{3-10}.7) =&  - (v_{a} u_{aXY}q, q_XX) - (v_{a} u_{aX}q_Y, q_XX)\\
&- (v_{a} u_{aY} q_X, q_XX) - (v_{a} u_a q_{XY}, q_XX)\\
=& O(1) \|v_a\|_\infty \|\widetilde{Y} u_{aXY}\|_\infty \|\frac{q}{\widetilde{Y}}\|_2 \|q_X\|_2 + O(1) \|v_a\|_\infty \| u_{aX}\|_\infty \|q_Y\|_2\|q_X\|_2\\
&- (v_{a} u_{aY} q_X, q_XX) + \frac{1}{2} (v_{aY} u_a q_{X}, q_XX) + \frac{1}{2} (v_{a} u_{aY} q_{X}, q_XX)\\
=& \frac{1}{2} ( (v_{aY} u_a- v_{a} u_{aY} ) q_{X}, q_XX) + O(\sqrt{\varepsilon})\|q_X\|_2\|q_Y\|_2,
\end{align*}
and
\begin{align*}
(\ref{3-10}.8) =& - (v_{a} u_{aYY}q, q_YX) - 2 (v_{a} u_{aY}q_Y, q_YX) - (v_{a} u_a q_{YY}, q_YX)\\
=& \frac{1}{2} (v_{aY} u_{aYY}q, qX) + \frac{1}{2} (v_{a} u_{aYYY}q, qX)\\
&- \frac{3}{2}(v_{a} u_{aY}q_Y, q_YX) + \frac{1}{2} (v_{aY} u_a q_Y, q_YX).
\end{align*}
Therefore,
\begin{equation}\label{3-10-1}
\begin{aligned}
(v_a \Delta \phi_Y , qX)=& \frac{1}{2} ((v_{a} u_{aYYY} -v_{aYYY} u_a) q, qX) + \frac{3}{2} ((v_{aY} u_a - v_{a} u_{aY})  q_Y, q_YX)\\
&+ \frac{1}{2} ( (v_{aY} u_a- v_{a} u_{aY} ) q_{X}, q_XX) +O(\sqrt{\varepsilon})\|q_X\|_2\|q_Y\|_2.
\end{aligned}
\end{equation}

Finally, we estimate the last term of $J[\phi]$. As in \eqref{2-10}, we have
\begin{equation}\label{3-11}
\begin{aligned}
(- \Delta v_s \phi_Y, qX) =& -(v_{aXX} \phi_Y, qX) -(v_{aYY} \phi_Y, qX)\\ 
&-(v_{rXX} \phi_Y, qX) -(v_{rYY} \phi_Y, qX)\\
=& -(v_{aYY} \phi_Y, qX)+ O( \sqrt\varepsilon + \|h\|_{\infty})\|q_X\|_2\|q_Y\|_2,
\end{aligned}
\end{equation}
and
\begin{align*}
-(v_{aYY} \phi_Y, qX)=& -(v_{aYY} u_{aY} q, qX) - (v_{aYY} u_a q_Y, qX)\\
=& \frac{1}{2} ((v_{aYYY} u_a -v_{aYY} u_{aY})  q, qX).
\end{align*}

\textbf{Step 4:} Cancellation.

Collecting all the terms from \eqref{3-1}-\eqref{3-11}, we have
\begin{equation}\label{3-12}
\begin{aligned}
&\frac{3}{2}(u_a^2 q_X, q_X)+ \frac{1}{2}(u_a^2 q_Y, q_Y) -2\varepsilon (u_a q_{XY}, q_{XY}X) + \varepsilon(u_{aYY}q_X,q_XX)\\
&-\varepsilon (u_a q_{YY}, q_{YY}X) + 2\varepsilon(u_{aYY}q_Y,q_YX) - \frac{1}{2} \varepsilon (u_{aYYYY}q, qX)\\
&+\varepsilon (u_{aY}q_{Y}, q_{Y}X)|_{Y=1} - \varepsilon (u_{aY}q_{Y}, q_{Y}X)|_{Y=-1}\\ 
&+ \frac{1}{2 } ((u_{aXYY} u_a - u_{aX}u_{aYY}) q,qX) + \frac{1}{2} ((v_{a} u_{aYYY} -v_{aYY} u_{aY}) q, qX)\\
& + \frac{3}{2} ((v_{aY} u_a - v_{a} u_{aY})  q_Y, q_YX) + \frac{1}{2} ( (v_{aY} u_a- v_{a} u_{aY} ) q_{X}, q_XX) 
\\=& (f,qX)+ O(\sqrt{\varepsilon} + \|h\|_\infty)(\|q_X\|_2 + \|q_Y\|_2).
\end{aligned}
\end{equation}
The terms $(\ref{3-12}.1)$, $(\ref{3-12}.2)$, $(\ref{3-12}.3)$, $(\ref{3-12}.5)$, $(\ref{3-12}.6)$, $(\ref{3-12}.8)$, and $(\ref{3-12}.9)$ are favorable nonnegative terms, which will be kept on the left-hand side. Note that from the construction of $u_p^0$, we have
\begin{equation}\label{Prandtl_2}
(u_e^0 + u_p^0) u_{pX}^0 + v_p^0 u_{pY}^0 - \varepsilon u_{pYY}^0 = O(\sqrt{\varepsilon}).
\end{equation}
By \eqref{concave},\eqref{Prandtl_2}, and $X \le 0$, we have
\begin{align*}
&(\ref{3-12}.4) + (\ref{3-12}.13)\\
=& \Big( (\varepsilon u_{aYY} + \frac{1}{2} v_{aY} u_a- \frac{1}{2} v_{a} u_{aY}) q_X, q_X X \Big)\\
=& \Big( (\varepsilon u^0_{pYY} - \frac{1}{2} u_{pX}^0 (u_e^0 + u_p^0)- \frac{1}{2} v_{p}^0 u_{pY}^0) q_X, q_X X \Big) + O(\sqrt{\varepsilon})\|q_X\|_2\\
\ge & \frac{1}{2} \Big( (\varepsilon u^0_{pYY} -u_{pX}^0 (u_e^0 + u_p^0)-  v_{p}^0 u_{pY}^0) q_X, q_X X \Big) + O(\sqrt{\varepsilon})\|q_X\|_2\\
=& O(\sqrt{\varepsilon})\|q_X\|_2.
\end{align*}
Similarly,
\begin{align*}
&(\ref{3-12}.6) + (\ref{3-12}.12)\\
=& \Big( (2\varepsilon u_{aYY} + \frac{3}{2} v_{aY} u_a- \frac{3}{2} v_{a} u_{aY}) q_Y, q_Y X \Big)\\
=& \Big( (2\varepsilon u^0_{pYY} - \frac{3}{2} u_{pX}^0 (u_e^0 + u_p^0)- \frac{3}{2} v_{p}^0 u_{pY}^0) q_Y, q_Y X \Big) + O(\sqrt{\varepsilon})\|q_Y\|_2\\
\ge & \frac{3}{2} \Big( (\varepsilon u^0_{pYY} -u_{pX}^0 (u_e^0 + u_p^0)-  v_{p}^0 u_{pY}^0) q_Y, q_Y X \Big) + O(\sqrt{\varepsilon})\|q_Y\|_2\\
=& O(\sqrt{\varepsilon})\|q_Y\|_2.
\end{align*}
For the rest terms, we use \eqref{Prandtl_2} and Hardy's inequality to obtain
\begin{align*}
&(\ref{3-12}.7) + (\ref{3-12}.10) + (\ref{3-12}.11)\\
=& \frac{1}{2} \Big( (-\varepsilon u_{aYYYY} + u_{aXYY} u_a - u_{aX}u_{aYY} + v_{a} u_{aYYY} -v_{aYY} u_{aY})q, qX \Big)\\
=& \frac{1}{2}\Big((-\varepsilon u_{aYY} + u_{aX}u_a + v_au_{aY})_{YY} q, qX\Big)\\
=&O(1) \| \widetilde{Y}^2  (-\varepsilon u_{aYY} + u_{aX}u_a + v_au_{aY})_{YY}\|_{\infty} \|\frac{q}{\widetilde{Y}}\|_2^2\\
=& O(\sqrt{\varepsilon}) \|q_Y\|_2^2.
\end{align*}
Therefore, by \eqref{Hardy_1} with $F = q_X$ and $q_Y$, \eqref{3-12} implies
\begin{align*}
&\frac{3}{2}\|u_a q_X\|_2^2 + \frac{1}{2}\|u_a q_Y\|_2^2\\ 
=& O(\sqrt{\varepsilon} + \|h\|_\infty) (\|q_X\|_2^2 + \|q_Y\|_2^2) + \|f\|_2^2\\
=& O(\sqrt{\varepsilon} + \|h\|_\infty) \Big( \sigma \varepsilon \|\sqrt{u_a} q_{XY}\|_2^2 + \frac{1}{\sigma^2}\|u_a q_X\|_2^2 + \sigma \varepsilon \|\sqrt{u_a} q_{YY}\|_2^2 + \frac{1}{\sigma^2}\|u_a q_Y\|_2^2 \Big)\\ &+ \|f\|_2^2.
\end{align*}
Choosing $\sigma = (\sqrt{\varepsilon} + \|h\|_\infty)^{\frac{1}{3}}$, we will obtain the desired estimate \eqref{energy_3}.
\end{proof}

\section{Proof of Theorems \ref{Thm1} and \ref{Thm2}}\label{sec_proof}

In this section, we prove Theorems \ref{Thm1} and \ref{Thm2}. Let us first recall the space $\cX$ and the corresponding norm defined in \eqref{def_spaceX}.
As a consequence of the estimates obtained in Section \ref{sec_est}, we have the following crucial apriori estimate.
\begin{proposition}\label{prop_estimate}
Under the assumptions of Theorem \ref{Thm1} or \ref{Thm2}, let $\phi$ be the solution of \eqref{remainder_linear} with the boundary conditions \eqref{remainder_boundary_condition}, then
$$
\|\phi\|_{\cX} \lesssim \|f\|_2.
$$
\end{proposition}
\begin{proof}
Under the assumptions of Theorem \ref{Thm1} or \ref{Thm2}, we use either Lemmas \ref{lem_1} and \ref{lem2} or Lemmas \ref{lem_1} and \ref{lem3} to close the estimates and obtain
$$
\| u_a q_{X}\|_2^2 + \| u_a q_{Y}\|_2^2 + \varepsilon \left( \| \sqrt{u_a} q_{XX}\|_2^2 +\| \sqrt{u_a} q_{XY}\|_2^2 + \| \sqrt{u_a} q_{YY}\|_2^2 \right) \lesssim  \|f\|_2^2.
$$ 
Then by \eqref{phi_X to q_X} and \eqref{Hardy_1} with $F=q_X$
$$
\|\phi_X\|_2 \lesssim \| q_X\|_2 \lesssim \| u_a q_{X}\|_2 + \sqrt\varepsilon \| \sqrt{u_a} q_{XY}\|_2 \lesssim \|f\|_2.
$$
Similarly, it is straightforward to see that
$$
\|\phi_Y\|_2 + \sqrt{\varepsilon}\|\nabla^2 \phi\|_2 \lesssim \|f\|_2.
$$
For $H^3$ estimate, we use an a-priori estimate for bi-Laplace operator on domains with right corners (see \cite{BR80}*{Theorem 2}), \eqref{u_a_property}, \eqref{u_r_property} and \eqref{u_r_property2} to obtain
\begin{align*}
\| \nabla^3 \phi\|_2 \lesssim& \frac{1}{\varepsilon} \|f\|_2 + \frac{1}{\varepsilon} (\|u_{sX}\|_\infty + \|v_{sY}\|_\infty)\|\nabla^2 \phi\|_2 \\
&+ \frac{1}{\varepsilon} (\|\Delta u_s\|_\infty + \|\Delta v_s\|_\infty) \|\nabla \phi\|_2\\
\lesssim& \frac{1}{\varepsilon^3} \|f\|_2.
\end{align*}
\end{proof}

Now we are ready to prove Theorems \ref{Thm1} and \ref{Thm2} using the contraction mapping method.

\begin{proof}[Proof of Theorems \ref{Thm1} and \ref{Thm2}]
The goal is to show the existence of a solution $\phi$ to \eqref{remainder_eq} satisfying $\|\phi\|_{\cX} \le 1$. We define a solution map $S: \psi \mapsto \phi$, where $\phi$ is the solution to
$$
u_s \Delta \phi_X - \Delta u_s \phi_X + v_s \Delta \phi_Y - \Delta v_s \phi_Y - \varepsilon \Delta^2 \phi = -\varepsilon^{-\frac{13}{2}} R - \varepsilon^{\frac{13}{2}} (\psi_Y \Delta \psi_X - \psi_X \Delta \psi_Y)
$$
with boundary conditions \eqref{remainder_boundary_condition}, where $R$ is given in \eqref{remainder_R}. Let
$$
B:= \{ \phi \in \cX: \|\phi\|_{\cX} < 1\}.
$$
We will prove that $S$ is a contraction mapping in $B$ when $\varepsilon$ is small. From the construction of approximated profile in Section \ref{sec_app}, we know that $R = O(\varepsilon^7)$. Then by Proposition \ref{prop_estimate}, we have
$$
\|S(\psi)\|_{\cX} \lesssim \varepsilon^{-\frac{13}{2}} R + \varepsilon^{\frac{13}{2}} \| \nabla \psi\|_\infty \| \nabla^3 \psi\|_2 \lesssim \varepsilon^{1/2}.
$$
Therefore, $S$ maps $B$ into $B$ when $\varepsilon$ is small. Let $\psi_1, \psi_2 \in B$, then we have
\begin{align*}
\| S(\psi_1 - \psi_2)\|_{\cX} \lesssim& \varepsilon^{\frac{13}{2}} \|\psi_{1Y} \Delta \psi_{1X} - \psi_{1X} \Delta \psi_{1Y} - \psi_{2Y} \Delta \psi_{2X} + \psi_{2X} \Delta \psi_{2Y}\|_2\\
\lesssim& \varepsilon^{\frac{13}{2}} \Big( \|\psi_{1Y} - \psi_{2Y} \|_\infty \| \nabla^3 \psi_{1}\|_2 + \|\psi_{2Y}\|_\infty \|\Delta \psi_{1X} - \Delta \psi_{2X}\|_2 \Big)\\
\lesssim& \varepsilon^{\frac{1}{2}} \|\psi_1 - \psi_2\|_{\cX}.
\end{align*}
Hence, $S$ is a contraction mapping in $B$ when $\varepsilon$ is small. Then we can conclude, using contraction mapping theorem that, there exists a unique solution $\phi$ to \eqref{remainder_eq} with the boundary condition \eqref{remainder_boundary_condition} satisfying $\|\phi\|_{\cX} < 1$. This concludes the proof of Theorems \ref{Thm1} and \ref{Thm2}.
\end{proof}

\bibliographystyle{amsplain}

\end{document}